\pdfoutput=1
\RequirePackage{ifpdf}
\ifpdf 
\documentclass[pdftex]{sigma}
\else
\documentclass{sigma}
\fi

\usepackage{mathrsfs}
\usepackage{dsfont}
\usepackage{tikz-cd}

\newtheorem{Theorem}{Theorem}[section]
\newtheorem*{Theorem*}{Theorem}
\newtheorem{Corollary}[Theorem]{Corollary}
\newtheorem{Lemma}[Theorem]{Lemma}
\newtheorem{Proposition}[Theorem]{Proposition}
{ \theoremstyle{definition}
\newtheorem{Definition}[Theorem]{Definition}

\newtheorem{Construction}[Theorem]{Construction}
\newtheorem{Example}[Theorem]{Example}
\newtheorem{Remark}[Theorem]{Remark} }

\numberwithin{equation}{section}

\begin{document}

\allowdisplaybreaks

\newcommand{\arXivNumber}{2403.03211}

\renewcommand{\PaperNumber}{092}

\FirstPageHeading

\ShortArticleName{Extension Theory and Fermionic Strongly Fusion 2-Categories}

\ArticleName{Extension Theory and Fermionic Strongly\\ Fusion 2-Categories
 (with an Appendix by\\ Thibault Didier D\'ecoppet and Theo Johnson-Freyd)}

\Author{Thibault Didier D\'ECOPPET}

\AuthorNameForHeading{T.D.~D\'ecoppet}

\Address{Mathematics Department, Harvard University, Cambridge, Massachusetts, USA}
\Email{\href{mailto:decoppet@math.harvard.edu}{decoppet@math.harvard.edu}}
\URLaddress{\url{https://www.thibaultdecoppet.com}}

\ArticleDates{Received March 11, 2024, in final form October 09, 2024; Published online October 17, 2024}

\Abstract{We study group graded extensions of fusion 2-categories. As an application, we~obtain a homotopy theoretic classification of fermionic strongly fusion 2-categories. We~examine various examples in detail.}


\Classification{18M20; 18N25}

\section{Introduction}

Let us fix an algebraically closed field $\mathds{k}$ of characteristic zero. In the theory of fusion 1-categories over $\mathds{k}$, one of the main tools for constructing new fusion 1-categories out of the ones that are already known is the concept of a group graded extension as introduced in \cite{ENO2}. Fusion 2-ca\-tegories over $\mathds{k}$ were introduced in \cite{DR}, and categorify the notion of a fusion 1-category. It is then only natural to investigate the notion of group graded extension for fusion 2-categories. More specifically, setting up group graded extension theory for fusion 2-categories is motivated by the problem of classifying fermionic strongly fusion 2-categories. In more detail, recall that a~bosonic (resp.\ fermionic) strongly fusion 2-category is a fusion 2-category whose braided fusion 1-category of endomorphisms of the monoidal unit is $\mathbf{Vect}$ (resp.\ $\mathbf{SVect}$). It was proven in~\cite{D9} that every fusion 2-category over $\mathds{k}$ is Morita equivalent to the 2-Deligne tensor product of a~strongly fusion 2-category and an invertible fusion 2-category, the latter corresponding to the data of a non-degenerate braided fusion 1-category. This motivates the problem of classifying strongly fusion 2-categories. Now, it was shown in \cite{JFY} that every simple object of such a fusion 2-category is invertible. In particular, strongly fusion 2-categories are group graded extensions of either $\mathbf{2Vect}$ or $\mathbf{2SVect}$. On the one hand, the classification of bosonic strongly fusion 2-categories is then routine: It consists of a finite group together with a 4-cocycle \cite[Section~2.1]{DR}. On the other hand, the classification of fermionic strongly fusion 2-categories is more subtle, and does require the full strength of extension theory for fusion 2-categories. We ought to mention that, at a Physical level of rigor and in the context of the classification of topological orders in~$(3+1)$ dimensions, related results were already discussed in \cite{LW} and \cite{JF}.

\subsection{Extension theory for fusion 1-categories}

Let us fix a fusion 1-category $\mathcal{C}$. We can then consider the associated space of invertible (finite semisimple) $\mathcal{C}$-$\mathcal{C}$-bimodule 1-categories, invertible bimodule functors, and bimodule natural isomorphisms. We denote this space by $\mathcal{B}{\rm r}\mathcal{P}{\rm ic}(\mathcal{C})$, and refer to it as the Brauer--Picard space of $\mathcal{C}$. This space admits a canonical product structure given by the relative Deligne tensor product over $\mathcal{C}$. In fact, extension theory was the original motivation for the introduction of the relative Deligne tensor product of finite semisimple module 1-categories over a fusion 1-category~\cite{ENO2}. The relative Deligne tensor product endows the Brauer--Picard space with the structure of a group-like topological monoid, so that we may consider its delooping $\mathrm{B}\mathcal{B}{\rm r}\mathcal{P}{\rm ic}(\mathcal{C})$. The main result of \cite{ENO2} is that, given any finite group $G$, faithfully $G$-graded extensions of $\mathcal{C}$ are parameterised by homotopy classes of maps from $\mathrm{B}G$ to $\mathrm{B}\mathcal{B}{\rm r}\mathcal{P}{\rm ic}(\mathcal{C})$.

\subsection{Results}

Over an algebraically closed field $\mathds{k}$ of characteristic zero, the existence of the relative 2-Deligne tensor product for finite semisimple module 2-categories over a fusion 2-category was established in \cite{D10}. In particular, for any fusion 2-category $\mathfrak{C}$, one can consider the associated space $\mathscr{B}{\rm r}\mathscr{P}{\rm ic}(\mathfrak{C})$ of invertible $\mathfrak{C}$-$\mathfrak{C}$-bimodule 2-categories. We refer to this space as the Brauer--Picard space of~$\mathfrak{C}$, and note that, by construction, it can be delooped. We obtain a categorification of the main theorem of \cite{ENO2} to fusion 2-categories, noting that a more general result was announced in \cite{JFR3}.

\begin{Theorem*}[Theorem~\ref{thm:extensiontheory}]
For any finite group $G$, faithfully $G$-graded extensions of the fusion~$2$-category $\mathfrak{C}$ are classified by homotopy classes of maps from $\mathrm{B}G$ to $\mathrm{B}\mathscr{B}{\rm r}\mathscr{P}{\rm ic}(\mathfrak{C})$, the delooping of the Brauer--Picard space of $\mathfrak{C}$.
\end{Theorem*}

We wish to remark that Tambara--Yamagami fusion 2-categories, which are $\mathbb{Z}/2$-graded fusion 2-categories whose non-trivially graded piece is $\mathbf{2Vect}$, were studied in \cite{DY23} via different methods. In fact, unlike in the decategorified setting, it seems difficult to use extension theory to classify Tambara--Yamagami fusion 2-categories, as the group $H^5(\mathrm{B}\mathbb{Z}/2;\mathds{k}^{\times})$, in which a certain obstruction lies, is non-zero.

Using results from \cite{GJF0} and \cite{JFR}, we perform a careful analysis of the structure of the Brauer--Picard space associated to the fusion 2-category $\mathbf{2SVect}$. When combined with the above theorem, it yields the following result.

\begin{Theorem*}[Theorem~\ref{thm:classificationFSF2Cs}]
Fermionic strongly fusion $2$-categories equipped with a faithful grading are classified by a finite group $G$ together with a class $\varpi$ in $H^2(\mathrm{B}G;\mathbb{Z}/2)$ and a class $\pi$ in~$SH^{4+\varpi}(\mathrm{B}G)$.
\end{Theorem*}

In the physics literature, versions of this classification have already appeared in \cite{LW}, and \cite[Section~V.D]{JF}. In particular, the notation $SH^{4+\varpi}(\mathrm{B}G)$ refers to $\varpi$-twisted supercohomology, the twisted generalized cohomology theory associated to the space $\mathbf{2SVect}^{\times}$ of invertible objects and morphisms in $\mathbf{2SVect}$. This generalized cohomology theory is well known in the literature on symmetry protected topological phases \cite{GJF0,KT,WG}, where it is referred to as extended supercohomology.

It is instructive to briefly discuss how the above data can be recovered from a given fermionic strongly fusion 2-category $\mathfrak{C}$. The corresponding group $G$ is the group of connected components of $\mathfrak{C}$. Further, the group of invertible objects of $\mathfrak{C}$ is a central extension of $G$ by $\mathbb{Z}/2$. However, it is not necessarily the extension corresponding to $\varpi$, rather it depends on the class $\pi$. We end by examining many examples of the above classification. In particular, we find in Example~\ref{ex:FSF2CZ/2} that there are exactly three fermionic strongly fusion 2-categories whose group of connected components is $\mathbb{Z}/2^n$ for any positive integer $n$. We also study the fermionic strongly fusion 2-categories whose group of connected components is $\mathbb{Z}/2\oplus\mathbb{Z}/2$ in Example~\ref{ex:FSF2CZ2Z2}. Finally, it was asked in \cite{JFY} whether there exists a finite group $G$ and a faithfully $G$-graded fermionic strongly fusion 2-category with trivial twist $\varpi$ whose group of invertible objects is a non-trivial central extension of $G$. In Appendix~\ref{sec:computation}, joint with Theo Johnson-Freyd, we perform a computation in supercohomology in order to show that this behaviour does occur when $G=\mathbb{Z}/4\oplus \mathbb{Z}/4$.

\section{Preliminaries}

\subsection{Fusion 2-categories}

Over an algebraically closed field $\mathds{k}$ of characteristic zero, the notions of finite semisimple 2-category and fusion 2-category were introduced in \cite{DR}. We succinctly review these definitions here, and refer the reader to the aforementioned reference and also \cite{D2} for details.

A ($\mathds{k}$-linear) 2-category $\mathfrak{C}$ is locally finite semisimple if its $\operatorname{Hom}$-1-categories are finite semisimple. An object $C$ of $\mathfrak{C}$ is called simple if $\operatorname{Id}_C$ is a simple object of the finite semisimple 1-category~$\operatorname{End}_{\mathfrak{C}}(C)$. Then, a \textit{finite semisimple} 2-category is a locally finite semisimple 2-category that is Cauchy complete in the sense of \cite{GJF} (see also \cite{D1}), has right and left adjoints for 1-morphisms, and has finitely many equivalence classes of simple objects. In a finite semisimple 2-category $\mathfrak{C}$, any two simple objects are called connected if there exists a non-zero 1-morphism between them. This defines an equivalence relation, and we write $\pi_0(\mathfrak{C})$ for the corresponding set of equivalences classes; This is the set of \textit{connected components} of $\mathfrak{C}$.

Now, the definition of a monoidal 2-category is well known. In particular, using the notations of \cite{SP}, a monoidal structure on the 2-category $\mathfrak{C}$ involves a 2-functor $\Box\colon \mathfrak{C}\times \mathfrak{C}\rightarrow\mathfrak{C}$, the monoidal product, and a distinguished object $I$, the monoidal unit, together with various other coherence data. Given any object $C$ of $\mathfrak{C}$, we can ask for the existence of its left dual $^{\sharp}C$ and its right dual~$C^{\sharp}$. These two notions were studied in \cite{Pstr}. Categorifying the concept of a fusion 1-category is therefore straightforward, and we obtain the following definition.

\begin{Definition}
A multifusion 2-category is a finite semisimple monoidal 2-category $\mathfrak{C}$ such that every object $C$ admits a right dual $C^{\sharp}$ and a left dual $^{\sharp}C$. A fusion 2-category is a multifusion 2-category whose monoidal unit $I$ is a simple object.
\end{Definition}

To every fusion 2-category $\mathfrak{C}$, there is an associated braided fusion 1-category $\Omega\mathfrak{C}:=\operatorname{End}_{\mathfrak{C}}(I)$, the 1-category of endomorphisms of the monoidal unit $I$ of $\mathfrak{C}$. Conversely, given any braided fusion 1-category $\mathcal{B}$, we can consider the fusion 2-category $\mathbf{Mod}(\mathcal{B})$ of finite semisimple left $\mathcal{B}$-module 1-categories with monoidal structure given by $\boxtimes_{\mathcal{B}}$, the relative Deligne tensor product. Below, we give another class of examples that will be relevant for our purposes. Many more may be found in \cite{DR} and \cite{DY23}, and others will be discussed subsequently.

\begin{Example}
Given $G$ a finite group, and $\pi$ a 4-cocycle for $G$ with coefficients in $\mathds{k}^{\times}$, we can then consider the fusion 2-category $\mathbf{2Vect}^{\pi}(G)$ of $G$-graded finite semisimple 1-categories, also known as 2-vector spaces, with pentagonator twisted by $\pi$. Such fusion 2-category were completely characterized in \cite{JFY} as those fusion 2-categories $\mathfrak{C}$ for which $\Omega\mathfrak{C}\simeq \mathbf{Vect}$. We call these bosonic strongly fusion 2-categories.
\end{Example}

\subsection{Module 2-categories and the relative 2-Deligne tensor product}

Let $\mathfrak{C}$ be a monoidal 2-category with monoidal product $\Box$, and monoidal unit $I$. A right $\mathfrak{C}$-module 2-category is a 2-category $\mathfrak{M}$ equipped with an action 2-functor $\Box\colon\mathfrak{M}\times\mathfrak{C}\rightarrow \mathfrak{M}$, and coherence data satisfying various axioms. Likewise, there is a notion of left $\mathfrak{C}$-module 2-category. Additionally, there are notions of $\mathfrak{C}$-module 2-functors, $\mathfrak{C}$-module 2-natural transformations, and~$\mathfrak{C}$-module modifications. There are also notions of bimodule 2-categories, and maps between them. We refer the reader to \cite[Section~2]{D4} for the precise definitions.

Let us now assume that $\mathfrak{C}$ is a fusion 2-category. If $\mathfrak{M}$ is a left $\mathfrak{C}$-module 2-category, then we can consider $\mathbf{\operatorname{End}}_{\mathfrak{C}}(\mathfrak{M})$, the monoidal 2-category of left $\mathfrak{C}$-module 2-endofunctors on $\mathfrak{M}$. The next result combines \cite[Theorem~5.3.2]{D8} with \cite[Corollary~5.1.3]{D9}.

\begin{Theorem}\label{thm:Moritadual}
The monoidal $2$-category $\mathbf{\operatorname{End}}_{\mathfrak{C}}(\mathfrak{M})$ is a multifusion $2$-category.
\end{Theorem}

Moreover, the multifusion 2-category $\mathbf{\operatorname{End}}_{\mathfrak{C}}(\mathfrak{M})$ is fusion if and only if $\mathfrak{M}$ is indecomposable as a finite semisimple left $\mathfrak{C}$-module 2-category.

For our purposes, we need to consider another operation on finite semisimple module 2-categories. Let $\mathfrak{M}$ be a finite semisimple right $\mathfrak{C}$-module 2-category and $\mathfrak{N}$ be a finite semisimple left $\mathfrak{C}$-module 2-category. There are notions of $\mathfrak{C}$-balanced 2-functors, $\mathfrak{C}$-balanced 2-natural transformations, and $\mathfrak{C}$-balanced modifications out of $\mathfrak{M}\times\mathfrak{N}$. We refer the reader to \cite[Section~2.2]{D10} for the details. The next result is a combination of \cite[Theorem~2.2.4]{D10} with \cite[Corollary~5.1.3]{D9}.

\begin{Theorem}
There is a $3$-universal $\mathfrak{C}$-balanced $2$-functor $\boxtimes_{\mathfrak{C}}\colon \mathfrak{M}\times\mathfrak{N}\rightarrow \mathfrak{M}\boxtimes_{\mathfrak{C}}\mathfrak{N}$ to a finite semisimple $2$-category.
\end{Theorem}

In fact, the proof of \cite[Theorem~2.2.4]{D10} gives an explicit description of $\mathfrak{M}\boxtimes_{\mathfrak{C}}\mathfrak{N}$. More precisely, if $A$ is a (necessarily separable) algebra in $\mathfrak{C}$ such that $\mathfrak{M}$ is equivalent to $\mathbf{LMod}_{\mathfrak{C}}(A)$, the right~$\mathfrak{C}$-module 2-category of left $A$-modules in $\mathfrak{C}$, and $B$ is a (necessarily separable) algebra in $\mathfrak{C}$ such that $\mathfrak{N}$ is equivalent to $\mathbf{Mod}_{\mathfrak{C}}(B)$, the left $\mathfrak{C}$-module 2-category of right $B$-modules in~$\mathfrak{C}$, then~$\mathfrak{M}\boxtimes_{\mathfrak{C}}\mathfrak{N}\simeq \mathbf{Bimod}_{\mathfrak{C}}(A,B)$, the finite semisimple 2-category of $A$-$B$-bimodules in $\mathfrak{C}$.

Finally, let us recall that, as was explained in \cite[Section~3.2]{D10}, the existence of the relative 2-Deligne tensor product allows us to consider the symmetric monoidal Morita 4-category $\mathbf{F2C}$ of (multi)fusion 2-categories, finite semisimple bimodule 2-categories, and bimodule morphisms.

\section{Extension theory}

For fusion 1-categories over an algebraically closed field of characteristic zero, extension theory was developed in \cite{ENO2}. The key concept is that of an invertible bimodule 1-category, or, more precisely, the space formed by such objects together with their invertible morphisms. Proceeding in a similar fashion, we will discuss extension theory for fusion 2-categories, i.e., we will study group graded fusion 2-categories in the sense of the definition below.

\begin{Definition}
Let $\mathfrak{C}$ be a fusion 2-category and $G$ a finite group. A $G$-grading on $\mathfrak{C}$ is a~decomposition $\boxplus_{g\in G}\mathfrak{C}_g$ of $\mathfrak{C}$ into a direct sum of finite semisimple 2-categories such that for every $C$ in $\mathfrak{C}_g$ and $D$ in $\mathfrak{C}_h$, $C\Box D$ lies $\mathfrak{C}_{gh}$. A $G$-grading on $\mathfrak{C}$ is faithful if $\mathfrak{C}_g$ is non-zero for every $g$ in $G$.
\end{Definition}

\subsection{Invertible bimodule 2-categories}

Let $\mathds{k}$ be an algebraically closed field of characteristic zero. We begin by recalling a definition from \cite{DY23}. Let $\mathfrak{C}$ and $\mathfrak{D}$ be two fusion 2-categories. We write $\mathfrak{D}^{\rm mop}$ for the fusion 2-category obtained by endowing the finite semisimple 2-category $\mathfrak{D}$ with the opposite of the monoidal product of $\mathfrak{D}$.

\begin{Definition}
A finite semisimple $\mathfrak{C}$-$\mathfrak{D}$-bimodule 2-category $\mathfrak{M}$ is invertible if the canonical monoidal 2-functor $\mathfrak{D}^{\rm mop}\rightarrow \mathbf{\operatorname{End}}_{\mathfrak{C}}(\mathfrak{M})$ is an equivalence.
\end{Definition}

At the decategorified level, i.e., for fusion 1-categories, invertibility of a finite semisimple bimodule 1-category admits various equivalent characterizations as is explained in \cite[Proposition~4.2]{ENO2}. A categorified version of this result was given in \cite{D10}, which we partially recall below.

\begin{Proposition}\label{prop:invertibilitycharacterization}
Let $\mathfrak{M}$ be a finite semisimple $\mathfrak{C}$-$\mathfrak{D}$-bimodule $2$-category. The following are equivalent:
\begin{enumerate}\itemsep=0pt
 \item[$(1)$] The $\mathfrak{C}$-$\mathfrak{D}$-bimodule $2$-category $\mathfrak{M}$ is invertible.
 \item[$(2)$] The $\mathfrak{C}$-$\mathfrak{D}$-bimodule $2$-category $\mathfrak{M}$ defines an invertible $1$-morphism from $\mathfrak{C}$ to $\mathfrak{D}$ in $\mathbf{F2C}$.
\end{enumerate}
\end{Proposition}

The next proposition gives preliminary insight into the relation between group graded fusion 2-categories and invertible bimodule 2-categories. We note that the first part of the result below has already appeared as \cite[Proposition~3.1.7]{DY23}. We also refer the reader to \cite[Theorem~6.1]{ENO2} for the decategorified version of this proposition.

\begin{Proposition}\label{prop:gradedMoritaequivalence}
Let $\mathfrak{C}$ be a faithfully $G$-graded fusion $2$-category $\mathfrak{C}$. For any $g\in G$, the finite semisimple $\mathfrak{C}_e$-$\mathfrak{C}_e$-bimodule $2$-category $\mathfrak{C}_g$ is invertible. Furthermore, for any $g,h\in G$, the $2$-functor $\Box\colon\mathfrak{C}\times\mathfrak{C}\rightarrow \mathfrak{C}$ induces an equivalence \smash{$\mathfrak{C}_g\boxtimes_{\mathfrak{C}_e}\mathfrak{C}_h\xrightarrow{\simeq}\mathfrak{C}_{gh}$} of $\mathfrak{C}_e$-$\mathfrak{C}_e$-bimodule $2$-categories.
\end{Proposition}
\begin{proof}
The first part follows from the second, so we only prove the latter. For the second part, note that it is enough to show that the canonical 2-functor is an equivalence. To this end, pick any non-zero objects $X$ in $\mathfrak{C}_{g^{-1}}$, and $Y$ in $\mathfrak{C}_{h}$. It follows from the proof of \cite[Theorem~5.4.3]{D8} that~$Y\Box {^{\sharp}Y}$ is a separable algebra in $\mathfrak{C}$, and that the left $\mathfrak{C}$-module 2-functor $\mathfrak{C}\rightarrow \mathbf{Mod}_{\mathfrak{C}}\big(Y\Box {^{\sharp}Y}\big)$ given by $C\mapsto C\Box {^{\sharp}Y}$ is an equivalence. Likewise, $X\Box {^{\sharp}X}$ is a separable algebra in $\mathfrak{C}$, and the right~$\mathfrak{C}$-module 2-functor $\mathfrak{C}\rightarrow \mathbf{LMod}_{\mathfrak{C}}\big(X\Box {^{\sharp}X}\big)$ given by $C\mapsto X\Box C$ is an equivalence. In particular, we also find that the 2-functor $\mathfrak{C}\rightarrow \mathbf{Bimod}_{\mathfrak{C}}\big(X\Box {^{\sharp}X},Y\Box {^{\sharp}Y}\big)$ given by $C\mapsto X\Box C\Box {^{\sharp}Y}$ is an equivalence. Under these identifications, as was recalled above, it follows from the proof of~\cite[Theorem~2.2.4]{D10} that the canonical 2-functor $\mathfrak{C}\times\mathfrak{C}\rightarrow\mathfrak{C}\boxtimes_{\mathfrak{C}}\mathfrak{C}$ is identified with $\Box\colon\mathfrak{C}\times\mathfrak{C}\rightarrow \mathfrak{C}$, which is given by
\begin{gather*}
 \mathbf{LMod}_{\mathfrak{C}}\big(X\Box {^{\sharp}X}\big)\times\mathbf{Mod}_{\mathfrak{C}}\big(Y\Box {^{\sharp}Y}\big) \rightarrow \mathbf{Bimod}_{\mathfrak{C}}\big(X\Box {^{\sharp}X},Y\Box {^{\sharp}Y}\big),\qquad
(M,N) \mapsto M\Box N.
\end{gather*}
 But, by considering the restriction $\mathfrak{C}_e\hookrightarrow\mathfrak{C}$, the above equivalences restrict to equivalences $\mathfrak{C}_h\rightarrow \mathbf{Mod}_{\mathfrak{C}_e}\big(Y\Box {^{\sharp}Y}\big)$, $\mathfrak{C}_g\rightarrow \mathbf{LMod}_{\mathfrak{C}_e}\big(X\Box {^{\sharp}X}\big)$, and $\mathfrak{C}_{gh}\rightarrow \mathbf{Bimod}_{\mathfrak{C}_e}\big(X\Box {^{\sharp}X},Y\Box {^{\sharp}Y}\big)$. In particular, the canonical 2-functor $\mathfrak{C}_g\times\mathfrak{C}_h\rightarrow\mathfrak{C}_g\boxtimes_{\mathfrak{C}_e}\mathfrak{C}_h$ is identified with $\Box:\mathfrak{C}_g\times\mathfrak{C}_h\rightarrow\mathfrak{C}_{gh}$ as claimed.
\end{proof}

A similar argument as the one used in the proof of the above proposition yields the following result.

\begin{Corollary}\label{cor:invertibledual}
The canonical $\mathfrak{C}_e$-$\mathfrak{C}_e$-bimodule $2$-functor $\mathfrak{C}_{g^{-1}}\rightarrow \mathbf{Fun}_{\mathfrak{C}_e}(\mathfrak{C}_g,\mathfrak{C}_e)$ given by $D\mapsto \{C\mapsto D\Box C\}$ is an equivalence.
\end{Corollary}

\subsection{Brauer--Picard spaces and extensions}

Group-graded extensions of fusion 1-categories are parameterised by maps into the space of invertible bimodule 1-categories and their (invertible) higher morphisms \cite{ENO2}. Thanks to the existence of the 4-category $\mathbf{F2C}$ obtained in \cite{D10}, the corresponding spaces for our 2-categorical purposes are easy to define.

\begin{Definition}
Let $\mathfrak{C}$ be a fusion 2-category. The Brauer--Picard space of $\mathfrak{C}$ consists of the invertible objects and the invertible morphisms in the monoidal 3-category of finite semisimple $\mathfrak{C}$-$\mathfrak{C}$-bimodule 2-categories, that is, \begin{gather*}\mathscr{B}{\rm r}\mathscr{P}{\rm ic}(\mathfrak{C}):=\big(\mathbf{\operatorname{End}}_{\mathbf{F2C}}(\mathfrak{C})\big)^{\times}.\end{gather*}
\end{Definition}

\begin{Remark}
We will write $Br\operatorname{Pic}(\mathfrak{C}):= \pi_0(\mathscr{B}{\rm r}\mathscr{P}{\rm ic}(\mathfrak{C}))$, the Brauer--Picard group of $\mathfrak{C}$. It follows from \cite[Lemma~2.2.1]{D9} that $\Omega \mathbf{\operatorname{End}}_{\mathbf{F2C}}(\mathfrak{C})\simeq \mathscr{Z}(\mathfrak{C})$, the Drinfeld center of $\mathfrak{C}$, as defined in~\cite{Cr}. Thus, the homotopy groups of the space $\mathscr{B}{\rm r}\mathscr{P}{\rm ic}(\mathfrak{C})$ are given as follows:
 $$\renewcommand{\arraystretch}{1.2} \begin{tabular}{|c|c|c|c|c|}
\hline
$\pi_0$ & $\pi_1$ & $\pi_2$ & $\pi_3$\\
\hline
${\rm Br}\operatorname{Pic}(\mathfrak{C})$ & $\operatorname{Inv}(\mathscr{Z}(\mathfrak{C}))$ & $\operatorname{Inv}(\Omega\mathscr{Z}(\mathfrak{C}))$ & $\mathds{k}^{\times}$\\
\hline
\end{tabular}$$
\end{Remark}

\begin{Example}\label{ex:BrPic2Vect}
Taking $\mathfrak{C}=\mathbf{2Vect}$, we have that $\mathscr{B}{\rm r}\mathscr{P}{\rm ic}(\mathfrak{C})$ is the core of the symmetric monoidal 3-category of multifusion 1-categories and the homotopy groups of this core are well known
 $$\renewcommand{\arraystretch}{1.2}\begin{tabular}{|c|c|c|c|c|}
\hline
$\pi_0$ & $\pi_1$ & $\pi_2$ & $\pi_3$\\
\hline
$1$ & $0$ & $0$ & $\mathds{k}^{\times}$\\
\hline
\end{tabular}$$
\end{Example}

\begin{Example}
More generally, we may also take $\mathfrak{C}=\mathbf{2Vect}_G^{\pi}$, in which case we have
 $$\renewcommand{\arraystretch}{1.2}\begin{tabular}{|c|c|c|c|c|}
\hline
$\pi_0$ & $\pi_1$ & $\pi_2$ & $\pi_3$\\
\hline \\[-1.1em]
$H^3(G;\mathds{k}^{\times})\rtimes \operatorname{Out}(G)$ & $Z(G)\oplus H^2(G;\mathds{k}^{\times})$ & $\smash{\widehat{G}}$ & $\mathds{k}^{\times}$\\
\hline
\end{tabular}$$
We have used \smash{$\widehat{G}$} to denote the group of multiplicative characters of $G$, \cite[Proposition~2.4.1]{DY23} for the description of $\pi_0$, and the main theorem of \cite{KTZ} for $\pi_1$ and $\pi_2$ together with the identification~$\operatorname{Pic}(\mathbf{Rep}(G))\cong H^2(G;\mathds{k}^{\times})$, which is given, for instance, in \cite[Proposition~6.1]{DN}.
\end{Example}

The following example will be relevant in the next section.

\begin{Example}\label{ex:BrauerPicardSymmetric}
Let $\mathcal{E}$ be a symmetric fusion 1-category, and take $\mathfrak{C}=\mathbf{Mod}(\mathcal{E})$. In this case we have
$$\renewcommand{\arraystretch}{1.2} \begin{tabular}{|c|c|c|c|c|}
\hline
$\pi_0$ & $\pi_1$ & $\pi_2$ & $\pi_3$\\
\hline \\[-1.1em]
$\smash{\widetilde{\mathcal{M}{\rm ext}}}(\mathcal{E})\rtimes \operatorname{Aut}^{\operatorname{br}}(\mathcal{E})$ & $Z(\mathrm{Spec}(\mathcal{E}))\oplus \operatorname{Pic}(\mathcal{E})$ & $\operatorname{Inv}(\mathcal{E})$ & $\mathds{k}^{\times}$\\
\hline
\end{tabular}
$$
We have used \smash{$\widetilde{\mathcal{M}{\rm ext}}(\mathcal{E})$} to denote the group of Witt-trivial minimal non-degenerate extensions of $\mathcal{E}$: The description of $\pi_0$ then follows from \cite[Corollary~3.1.7]{D9}. The statement for $\pi_2$ follows from of \cite[Lemma~2.16]{JFR}. Finally, the description of $\pi_1$ follows from \cite[Corollary~6.11]{DN}. Let us also point out that the group $\operatorname{Pic}(\mathcal{E})$ with $\mathcal{E}$ super-Tannakian was computed in \cite[Theorem~6.5]{DN}.
\end{Example}

A more general version of the next result has been announced \cite{JFR3}. We are very grateful to them for outlining their proof to us, which has inspired the argument that we give below.

\begin{Theorem}\label{thm:extensiontheory}
Let $G$ be a finite group, and $\mathfrak{C}$ be a fusion $2$-category. Then, faithfully $G$-graded extensions of $\mathfrak{C}$ are classified by homotopy classes of maps $\mathrm{B}G\rightarrow \mathrm{B}\mathscr{B}{\rm r}\mathscr{P}{\rm ic}(\mathfrak{C})$.
\end{Theorem}
\begin{proof}
For the purpose of this proof, it will be convenient to think of maps of spaces $\mathrm{B}G\rightarrow \mathrm{B}\mathscr{B}{\rm r}\mathscr{P}{\rm ic}(\mathfrak{C})$ as monoidal maps $G\rightarrow \mathscr{B}{\rm r}\mathscr{P}{\rm ic}(\mathfrak{C})$ of $(\infty,1)$-categories. We will also consider the $(\infty,1)$-category $\mathscr{C}$ obtained from $\mathbf{\operatorname{End}}_{\mathbf{F2C}}(\mathfrak{C})$ by only considering the invertible $n$-morphisms when $n\geq 2$. Then, by definition we have $\mathscr{B}{\rm r}\mathscr{P}{\rm ic}(\mathfrak{C})=\mathscr{C}^{\times}$ as monoidal $(\infty,1)$-categories.

We begin the proof by some general nonsense. Recall that algebras in the monoidal $(\infty,1)$-category $\mathscr{C}$ correspond precisely to lax monoidal functors $*\rightarrow \mathscr{C}$ (see \cite[Example~2.2.6.10]{L2}). We want a similar description for the notion of a (faithfully) $G$-graded algebra. In order to do so, consider the monoidal 1-category $G^{\sqcup}$, which is the coproduct completion of $G$. Now, there is a canonical algebra $A[G]$ in $G^{\sqcup}$ given by
 \begin{gather*}
A[G]:=\coprod_{g\in G}g,
\end{gather*}
 or equivalently a lax monoidal functor $*\rightarrow G^{\sqcup}$. Then, giving a $G$-grading on an algebra $A$ in $\mathscr{C}$ is equivalent to providing a factorization of lax monoidal functors
\begin{displaymath}
\begin{tikzcd}
* \arrow[rr,"A"] \arrow[rd, "A\lbrack G \rbrack"'] & & \mathscr{C} \\
 & G^{\sqcup}. \arrow[ru] &
\end{tikzcd}\end{displaymath}
But, given that $\mathscr{C}$ has direct sums, lax monoidal functors $G^{\sqcup}\rightarrow \mathscr{C}$ correspond exactly to lax monoidal functors $G\rightarrow \mathscr{C}$.

Now, given a map of spaces $\mathrm{B}G\rightarrow \mathrm{B}\mathscr{B}{\rm r}\mathscr{P}{\rm ic}(\mathfrak{C})$, or equivalently a strongly monoidal functor $F\colon G\rightarrow \mathscr{C}$, we obtain a $G$-graded algebra $\mathfrak{D}:=F(A[G])$ in $\mathfrak{C}$. We also write
\begin{gather*}
\mathfrak{D}=\boxplus_{g\in G}\mathfrak{D}_g=\boxplus_{g\in G}F(g).
\end{gather*}
 In fact, as $F(e) = \mathfrak{C}$ by definition, the monoidal 2-category $\mathfrak{D}$ is a faithfully $G$-graded extension of $\mathfrak{C}$. It is therefore enough to prove that $\mathfrak{D}$ is a fusion 2-category. The only property that is not obvious is rigidity. We have that $\mathbf{\operatorname{End}}_{\mathfrak{C}}(\mathfrak{D})$ is a fusion 2-category thanks to Theorem~\ref{thm:Moritadual} above. Let us fix $g\in G$. For any simple object $X$ in $\mathfrak{D}_g$, we can consider the left $\mathfrak{C}$-module 2-functor~${R_X\colon \mathfrak{D}\rightarrow \mathfrak{D}}$ given by $D\mapsto D\Box X$ on ${\mathfrak{D}_e\simeq \mathfrak{C}}$, and zero on $\mathfrak{D}_h$ for any~${e\neq h\in G}$. As the objects of the monoidal 2-category $\mathbf{\operatorname{End}}_{\mathfrak{C}}(\mathfrak{D})$ have duals, we can consider the right adjoint~$R_X^*$ of $R_X$. But, we have that $\mathfrak{D}_{g^{-1}}\rightarrow \mathbf{Fun}_{\mathfrak{D}_e}(\mathfrak{D}_g,\mathfrak{D}_e)$ is an equivalence of $\mathfrak{D}_e$-$\mathfrak{D}_e$-bimodule 2-categories as~$F$ is strongly monoidal, and therefore preserves duals. Thus, we find that ${R_X^*\simeq Y\Box (-)\colon \mathfrak{D}_g\rightarrow \mathfrak{D}_e=\mathfrak{C}}$ for some $Y$ in $\mathfrak{D}_{g^{-1}}$. This proves that $X$ has a left dual $Y$ in~$\mathfrak{D}$. One shows analogously that $X$ has a right dual.

Conversely, given $\mathfrak{D}=\boxplus_{g\in G}\mathfrak{D}_g$ a $G$-graded extension of $\mathfrak{C}= \mathfrak{D}_e$, we can consider the corresponding lax monoidal functor $F\colon G\rightarrow \mathscr{C}$. More precisely, we set $F(g) = \mathfrak{D}_g$ with lax monoidal structure given by the monoidal structure of $\mathfrak{D}$. We want to show that $F$ is strongly monoidal and factors through $\mathscr{B}{\rm r}\mathscr{P}{\rm ic}(\mathfrak{C})\subset \mathscr{C}$. It follows from the definition of an extension that $F$ is strongly unital. Proposition \ref{prop:gradedMoritaequivalence} above establishes that $F$ has the remaining desired properties.
\end{proof}

\begin{Remark}
The classification of Tambara--Yamagami 1-categories may be recovered from extension theory for fusion 1-categories as explained in \cite[Section~9.2]{ENO2}. It would similarly be interesting to understand the classification of Tambara--Yamagami 2-categories obtained in \cite[Proposition~5.2.3]{DY23} in terms of the extension theory of fusion 2-categories. Namely, Tambara--Yamagami 2-categories are by definition $\mathbb{Z}/2$-graded extensions of $\mathbf{2Vect}(A[1]\times A[0])$ by $\mathbf{2Vect}$ for some finite abelian group $A$. However, there is a complication that arises with the case of Tambara--Yamagami 2-categories: The group $H^5(\mathrm{B}\mathbb{Z}/2;\mathds{k}^{\times})\cong\mathbb{Z}/2$ is non-zero. But, in order to construct a map of spaces $\mathrm{B}\mathbb{Z}/2\rightarrow \mathrm{B}\mathscr{B}{\rm r}\mathscr{P}{\rm ic}(\mathbf{2Vect}(A[1]\times A[0]))$ one has to check that a certain obstruction class living in $H^5(\mathrm{B}\mathbb{Z}/2;\mathds{k}^{\times})$ vanishes. We do not know how to do this directly. In fact, the vanishing of such obstructions is a well-known difficulty in the extension theory of fusion 1-categories. However, for Tamabra--Yamagami 1-categories the relevant group~${H^4(\mathrm{B}\mathbb{Z}/2;\mathds{k}^{\times})=0}$ is trivial. Relatedly, general results guaranteeing the vanishing of this obstruction for fusion 1-categories are known such as \cite[Theorem~8.16]{ENO2}. We wonder whether such a criterion may be established for extensions of higher fusion categories.
\end{Remark}

\begin{Remark}
Over an arbitrary field, some results on extension theory for finite semisimple tensor 1-categories were obtained in \cite{San}. Going up one categorical level, one can setup extension theory for \textit{locally separable} compact semisimple tensor 2-categories over an arbitrary field. Namely, the reference \cite{D10} does work at this level of generality, and the proofs of both Proposition~\ref{prop:gradedMoritaequivalence} and Theorem~\ref{thm:extensiontheory} continue to hold up to the obvious modifications.
\end{Remark}

\subsection{Extensions from crossed braided fusion 1-categories}

We now review \cite[Construction~2.1.23]{DR}, which will later allow us to give very concrete models for some fermionic strongly fusion 2-categories. A related construction appeared in \cite[Section~6]{Cui}. In a slightly different direction, we also refer the reader to \cite{JPR} for a detailed discussion of the relation between $G$-crossed braided 1-categories and higher categories.

\begin{Construction}\label{con:gradedF2CfromCBF1C}
Fix $G$ a finite group, and let $\mathcal{C}$ be a (not necessarily faithfully graded) $G$-crossed braided fusion 1-category. Following the notations of \cite[Section~8.24]{EGNO}, $\mathcal{C}$ is a~(not necessarily faithfully) $G$-graded fusion 1-category $\mathcal{C}=\boxplus_{g\in G}\mathcal{C}_g$ equipped with a $G$-action~${g\mapsto T_g}$ such that $T_g(\mathcal{C}_h)\subseteq \mathcal{C}_{ghg^{-1}}$ together with suitably coherent natural isomorphisms $c_{W,X}\colon W\otimes X\cong T_g(X)\otimes W$ whenever $W$ is in $\mathcal{C}_g$. For simplicity, and without loss of generality, we will assume that the underlying monoidal 1-category of $\mathcal{C}$ is strict.

We can then consider the monoidal 2-category $\widehat{\mathfrak{D}}$ whose set of objects is the finite group~$G$, and with $\operatorname{Hom}$-1-categories given by $\operatorname{Hom}_{\widehat{\mathfrak{D}}}(g,h):=\mathcal{C}_{hg^{-1}}$.
 Composition of 1-morphisms is given by the tensor product $\otimes$ of $\mathcal{C}$. Then, the monoidal 2-functor~$\Box$ is defined by
 \begin{gather*}
\Box\colon\ \operatorname{Hom}_{\widehat{\mathfrak{D}}}(g_1,h_1)\times \operatorname{Hom}_{\widehat{\mathfrak{D}}}(g_2,h_2)\rightarrow \operatorname{Hom}_{\widehat{\mathfrak{D}}}(g_1g_2,h_1h_2),\qquad (W,X) \mapsto W\otimes T_{g_1}(X)
\end{gather*}
and its naturality constraints are given by the $G$-crossed braided structure of $\mathcal{C}$. More precisely, for any 1-morphisms $W\colon g_1\rightarrow h_1$, $X\colon g_2\rightarrow h_2$, $Y\colon h_1\rightarrow f_1$, $Z\colon h_2\rightarrow f_2$, the interchanger $\phi^{\Box}$ is given by
\begin{gather*}
\smash{\phi^{\Box}_{(W,X),(Y,Z)}\colon\ (YT_{h_1}(Z))(WT_{g_1}(X))\xrightarrow{c^{-1}_{W,Z}}YWT_{g_1}(Z)T_{g_1}(X)\xrightarrow{(\mu_{g_1})_{Z,X}}(YW)T_{g_1}(ZX)}.
\end{gather*}
Then, the associator 2-natural equivalence $\alpha$ is given on objects by $\alpha_{g_1,g_2,g_2}=\operatorname{Id}_{g_1g_2g_3}$, and on 1-morphisms $W\colon g_1\rightarrow h_1$, $X\colon g_2\rightarrow h_2$, $Y\colon g_3\rightarrow h_3$ by
\begin{gather*}
\smash{\alpha_{W,X,Y}\colon\ WT_{g_1}(X)T_{g_1g_2}(Y)\xrightarrow{(\gamma^{-1}_{g_1,g_2})_Y}WT_{g_1}(X)T_{g_1}(T_{g_2}(Y))\xrightarrow{(\mu_{g_1})_{X,Y}}WT_{g_1}(XT_{g_2}(Y))}.
\end{gather*}
Likewise, the 2-natural equivalences witnessing unitality are the obvious ones. The pentagonator as well as the other invertible modifications that have to be specified are all taken to be the identity ones. That these assignments yield a monoidal 2-category follow at once from the axioms of a $G$-crossed braided 1-category. Finally, we obtain a fusion 2-category $\mathfrak{D}$ by taking the Cauchy completion of $\widehat{\mathfrak{D}}$ in the sense of~\cite{GJF} (see also \cite{D1,DR}).
\end{Construction}

\begin{Remark}
Let us write $\mathrm{supp}_G(\mathcal{C})$ for the support of $\mathcal{C}$ in $G$, that is the subset of elements~${g\in G}$ such that $\mathcal{C}_g$ is non-zero. As $\mathcal{C}$ is $G$-crossed braided, $\mathrm{supp}_G(\mathcal{C})$ is a normal subgroup of $G$. It follows from the above construction that $\pi_0(\mathfrak{D})=G/\mathrm{supp}_G(\mathcal{C})$ as finite groups, and that $\mathfrak{D}$ is faithfully $G/\mathrm{supp}_G(\mathcal{C})$-graded.
\end{Remark}

\section{Fermionic strongly fusion 2-categories}\label{sec:FSF2Cs}

Recall that $\mathds{k}$ is an algebraically closed field of characteristic zero. It is well known that bosonic strongly fusion 2-categories are classified by a finite group $G$ and a 4-cocycle for $G$ with coefficients in $\mathds{k}^{\times}$. In fact, once we know that every simple object of such a fusion 2-category is invertible, as follows from \cite[Theorem~A]{JFY}, the problem becomes a straightforward application of Theorem~\ref{thm:extensiontheory} with Example~\ref{ex:BrPic2Vect}. At a Physical level of rigor, this was already observed in~\cite{LKW}.

Below, we will establish a similar classification of fermionic strongly fusion 2-categories. Recall that a \textit{fermionic strongly fusion} 2-category is a fusion 2-category $\mathfrak{C}$ such that $\Omega\mathfrak{C}=\mathbf{SVect}$. Then, it follows from \cite[Theorem~B]{JFY} that every simple object of $\mathfrak{C}$ is invertible. This was first observed in the physics literature \cite{LW}. As a consequence of this last fact, we find that $\pi_0(\mathfrak{C})$ is a finite group, and there is a central extension of finite groups
\begin{gather*}
0\rightarrow \mathbb{Z}/2\rightarrow \operatorname{Inv}(\mathfrak{C})\rightarrow \pi_0(\mathfrak{C})\rightarrow 1.
\end{gather*}
 Namely, we have $\operatorname{Inv}(\mathbf{2SVect})\cong \mathbb{Z}/2$. Examples are known for which the associated short exact sequence is not split (see \cite[Example 2.1.27]{DR} or Figure \ref{fig:exoticFSF2C} below).

\subsection[The Brauer--Picard space of 2SVect]{The Brauer--Picard space of $\mathbf{2SVect}$}\label{sub:fermionicclassification}

It follows from Example~\ref{ex:BrauerPicardSymmetric} above that the homotopy groups of the space $\mathscr{B}{\rm r}\mathscr{P}{\rm ic}(\mathbf{2SVect})$ are as follows $$\renewcommand{\arraystretch}{1.2} \begin{tabular}{|c|c|c|c|c|}
\hline
$\pi_0$ & $\pi_1$ & $\pi_2$ & $\pi_3$\\
\hline
$1$ & $\mathbb{Z}/2\oplus\mathbb{Z}/2$ & $\mathbb{Z}/2$ & $\mathds{k}^{\times}$\\
\hline
\end{tabular}$$
In order to completely characterize the space $\mathrm{B}\mathscr{B}{\rm r}\mathscr{P}{\rm ic}(\mathbf{2SVect})$, we have to understand its Postnikov $k$-invariants. They are precisely those of the braided monoidal 2-category $\mathscr{Z}(\mathbf{2SVect})^{\times}$, that is of the space $\mathrm{B}^2\mathscr{Z}(\mathbf{2SVect})^{\times}$, and these $k$-invariants can be determined using \cite{GJF0,JF2}. In particular, this space has non-zero homotopy groups exactly in degrees $2$, $3$, and $4$. So as to do this, we review some notation; For the most part we follow those used in \cite[Section~3.1]{JFR}.
\begin{itemize}\itemsep=0pt
 \item Given an abelian group $A$ and a non-negative integer $n$, we use $A[n]$ to denote the $n$-th Eilenberg--MacLane space associated to $A$.
 \item The cohomology ring $H^{\bullet}(\mathbb{Z}/2[n];\mathbb{Z}/2)$ is generated by an element in degree $n$ under cup products and the action of the Steenrod operations $\mathrm{Sq}^i$.
 \item We will also consider the cohomology groups $H^{\bullet}(\mathbb{Z}/2[n];\mathds{k}^{\times})$. Via the map $t\mapsto (-1)^t$ induced by the inclusion $\mathbb{Z}/2\hookrightarrow \mathds{k}^{\times}$, we can give names to all of the cohomology classes that we will consider in these groups.
\end{itemize}

\begin{Lemma}\label{lem:kinvariantsZ2SVect}
Let us write $c_2$ for the generator of $H^2((\mathbb{Z}/2\oplus 0)[2];\mathbb{Z}/2)$ and $m_2$ for the generator of $H^2((0\oplus \mathbb{Z}/2)[2];\mathbb{Z}/2)$. The first $k$-invariant of $\mathrm{B}^2\mathscr{Z}(\mathbf{2SVect})^{\times}$ is $c_2^2+c_2m_2$ in the group $H^4((\mathbb{Z}/2\oplus\mathbb{Z}/2)[2];\mathbb{Z}/2)$. Let us write \begin{gather*}
\smash{Y:=\mathrm{Fib}\big((\mathbb{Z}/2\oplus\mathbb{Z}/2)[2]\xrightarrow{c_2^2+c_2m_2}\mathbb{Z}/2[4]\big)},
\end{gather*}
 for the $($homotopy$)$ fiber, and $t_3$ for the generator of $H^3(\mathbb{Z}/2[3];\mathbb{Z}/2)$, the second $k$-invariant of~$\mathrm{B}^2\mathscr{Z}(\mathbf{2SVect})^{\times}$, which lives in $H^5(Y;\mathds{k}^{\times})$, is
$\sigma:=(-1)^{\mathrm{Sq}^2t_3 + t_3m_2}$.
\end{Lemma}
\begin{proof}
The underlying 2-category of $\mathscr{Z}(\mathbf{2SVect})$ is depicted below (see, for instance, \cite{JFR}). $$\begin{tikzcd}
I \arrow[d, "\mathbf{Vect}"', bend right] \arrow["\mathbf{SVect}"', loop, distance=2em, in=125, out=55] & & M \arrow[d, "\mathbf{Vect}"', bend right] \arrow["\mathbf{Vect}_{\mathbb{Z}/2}"', loop, distance=2em, in=125, out=55] \\
C \arrow[u, "\mathbf{Vect}"', bend right] \arrow["\mathbf{Vect}_{\mathbb{Z}/2}"', loop, distance=2em, in=305, out=235] & & C\Box M \arrow[u, "\mathbf{Vect}"', bend right] \arrow["\mathbf{Vect}_{\mathbb{Z}/2}"', loop, distance=2em, in=305, out=235]
\end{tikzcd}$$ We want to describe the space $\mathrm{B}^2\mathscr{Z}(\mathbf{2SVect})^{\times}$. We begin by determining its first $k$-invariant. We have $\mathscr{Z}(\mathbf{2SVect})^0\simeq \mathbf{2SVect}$ as symmetric fusion 2-categories, so the restriction of the first $k$-invariant to the braided monoidal sub-2-category spanned by $C$ has to be $c_2^2$ by \cite{GJF0}. Furthermore, its restriction to both $M$ and $C\Box M$ has to be trivial by \cite[Theorem~3.2]{JFR} and its proof. This uniquely determines the first $k$-invariant as $c_2^2+c_2m_2$ in $H^4((\mathbb{Z}/2\oplus \mathbb{Z}/2)[2];\mathbb{Z}/2)$.

Now, let $Y$ be the space in the statement of the lemma. We want to determine the group $H^5(Y;\mathds{k}^{\times})$ and its generators. In order to do so, we use the Serre spectral sequence for $Y$
\begin{gather*}
E_2^{i,j}=H^i((\mathbb{Z}/2\oplus\mathbb{Z}/2)[2], H^j(\mathbb{Z}/2[3];\mathds{k}^{\times}))\Rightarrow H^{i+j}(Y;\mathds{k}^{\times}).
\end{gather*}
 The first entries of the $E_2$ page for this spectral sequence are given by
\begin{equation}\renewcommand{\arraystretch}{1.2}\label{eq:E2braided}
 \begin{array}{c|cccccccc}
 j \\
 5 & \mathbb{Z}/2 & 0 \\
 4 & 0 & 0 & 0 \\
 3 & \mathbb{Z}/2 & 0 & \mathbb{Z}/2^{\oplus 2} & \mathbb{Z}/2^{\oplus 2}\\
 2 & 0 & 0 & 0 & 0 & 0 \\
 1 & 0 & 0 & 0 & 0 & 0 & 0 \\
 0 & \mathds{k}^{\times}& 0 & \mathbb{Z}/2^{\oplus 2} & 0 & \mathbb{Z}/2^{\oplus 2}\oplus\mathbb{Z}/4& \mathbb{Z}/2^{\oplus 3} &\mathbb{Z}/2^{\oplus 4}\\
 \hline
 & 0 & 1 & 2 & 3 & 4 & 5 & 6 & \quad i.
\end{array}
\end{equation}

We will abstain from giving generators for all the non-zero entries. Let us nonetheless point out that the $(6,0)$ entry is generated by
\begin{gather*}
(-1)^{c_2^3},\qquad (-1)^{c_2^2m_2}=(-1)^{c_2m_2^2},\qquad (-1)^{m_2^3},\qquad (-1)^{\mathrm{Sq}^1c_2 \mathrm{Sq}^1m_2}.
\end{gather*}
 The equality follows from the fact that $(-1)^{\mathrm{Sq}^1}$ is the trivial map. Now, the $(5,0)$ entry automatically survives to $E_{\infty}$ as no differential can hit it. Then, the $d_2$ differential is trivial, and the~$d_3$ differential is given by $d_3\big((-1)^{t_3}\big)=(-1)^{c_2^2+c_2m_2}$ due to the $k$-invariant of $Y$. In particular, we find that
 \begin{gather*}
 d_3\big((-1)^{t_3c_2}\big)=(-1)^{c_2^3+c_2^2m_2},\qquad d_3\big((-1)^{t_3m_2}\big)=(-1)^{c_2^2m_2+c_2m_2^2}=0,
 \end{gather*}
 so that the $\mathbb{Z}/2$ summand of the $(2,3)$ entry that is generated by $(-1)^{t_3m_2}$ does survive to the~$E_{\infty}$ page. It remains to understand what happens to the entry in degree $(0,5)$. We will argue that this entry survives. Namely, we have that the space $Y$ is a truncation of~$\mathrm{B}^2\mathscr{Z}(\mathbf{2SVect})^{\times}$. But, we know that $\Omega\mathscr{Z}(\mathbf{2SVect})^{\times}=\mathbf{SVect}^{\times}$, and it is well known that the $k$-invariant of the space $\mathrm{B}^3\mathbf{SVect}^{\times}$ is the non-trivial class $(-1)^{\mathrm{Sq}^2t_3}$ in $H^5(\mathbb{Z}/2[3];\mathds{k}^{\times})$. Thus, the edge map~$H^5(Y;\mathds{k}^{\times})\rightarrow H^5(\mathbb{Z}/2[3];\mathds{k}^{\times})$ induced by the inclusion $\mathbb{Z}/2[3]\hookrightarrow Y$ is non-zero, and the claim follows. Therefore, we find that $H^5(Y;\mathds{k}^{\times})$ is generated by $(-1)^{\mathrm{Sq}^2t_3}$, $(-1)^{t_3m_2}$, $(-1)^{\mathrm{Sq}^2\mathrm{Sq}^1c_2}$, $(-1)^{c_2\mathrm{Sq}^1m_2}=(-1)^{m_2\mathrm{Sq}^1c_2}$, and $(-1)^{\mathrm{Sq}^2\mathrm{Sq}^1m_2}=(-1)^{m_2\mathrm{Sq}^1m_2}$.

We now turn our attention towards describing the second $k$-invariant $\sigma$ of $\mathrm{B}^2\mathscr{Z}(\mathbf{2SVect})^{\times}$. To this end, let us write
\begin{gather*}
X:=\mathrm{Fib}\big(\mathbb{Z}/2[2]\xrightarrow{0}\mathbb{Z}/2[4]\big).
\end{gather*}
 Then, there is an inclusion $f\colon X\hookrightarrow Y$ induced by the inclusion of the second summand $\mathbb{Z}/2[2]\hookrightarrow (\mathbb{Z}/2\oplus\mathbb{Z}/2)[2]$. (This corresponds to picking out the object $M$ of $\mathscr{Z}(\mathbf{2SVect})^{\times}$.) Let us note that this map is compatible with the spectral sequence \eqref{eq:E2braided} above. It was computed in \cite[Theorem~3.2]{JFR} that the pullback $f^*(\sigma) = (-1)^{\mathrm{Sq}^2t_3 + t_3m_2}$, so that $\sigma$ must contain at least the factors $(-1)^{\mathrm{Sq}^2t_3}$ and $(-1)^{t_3m_2}$, and cannot contain $(-1)^{\mathrm{Sq}^2\mathrm{Sq}^1m_2}$. In addition, there is another map $g\colon X\hookrightarrow Y$ induced by the diagonal inclusion $\mathbb{Z}/2[2]\hookrightarrow (\mathbb{Z}/2\oplus\mathbb{Z}/2)[2]$. (This corresponds to picking out the object $C\Box M$ of $\mathscr{Z}(\mathbf{2SVect})^{\times}$.) Again, it follows from \cite[Theorem~3.2]{JFR} that the pullback $g^*(\sigma) = (-1)^{\mathrm{Sq}^2t_3 + t_3m_2}$. But, by naturality of the Serre spectral sequence, we find that $g^*\big((-1)^{\mathrm{Sq}^2\mathrm{Sq}^1c_2}\big) = (-1)^{\mathrm{Sq}^2\mathrm{Sq}^1t_2} = g^*\big((-1)^{c_2\mathrm{Sq}^1m_2}\big)$ with $t_2$ the generator of $H^2(\mathbb{Z}/2[2];\mathbb{Z}/2)$, so that
 \begin{gather*}
 \sigma = (-1)^{\mathrm{Sq}^2t_3 + t_3m_2}\qquad \text{or}\qquad \sigma = (-1)^{\mathrm{Sq}^2t_3 + t_3m_2 + \mathrm{Sq}^2\mathrm{Sq}^1c_2 + c_2\mathrm{Sq}^1m_2}.
 \end{gather*}
 It turns out that these two possibilities are equivalent. More precisely, consider the homotopy autoequivalence $\phi\colon Y\simeq Y$ with $\phi^*(t_3) = t_3 + \mathrm{Sq}^1c_2$. Then, we have
 \begin{gather*}
 \phi^*\big((-1)^{\mathrm{Sq}^2t_3 + t_3m_2}\big) = (-1)^{\mathrm{Sq}^2t_3 + t_3m_2 + \mathrm{Sq}^2\mathrm{Sq}^1c_2 + c_2\mathrm{Sq}^1m_2},
 \end{gather*}
which concludes the proof.
\end{proof}

\subsection{Twisted supercohomology}

In order to be able to carry out computations, as well as to make contact with the existing literature, it is useful to unpack the data of the space $\mathscr{B}{\rm r}\mathscr{P}{\rm ic}(\mathbf{2SVect})$ in a different way and express the classification of fermionic strongly fusion 2-categories using supercohomology.

\begin{Definition}
Supercohomology is the generalized cohomology theory associated to the double loop space of the spectrum associated to $\mathbf{2SVect}^{\times}$. For any space $X$ and integer $n$, we write~$SH^n(X)$ for the group of homotopy classes of maps of spaces from $X$ to $\mathrm{B}^{n-2}\mathbf{2SVect}^{\times}$.
\end{Definition}

This generalized cohomology theory first appeared in the literature on symmetry protected topological phases \cite{GJF0, KT,WG}, where it is also referred to as extended supercohomology.\footnote{An earlier but distinct notion of supercohomology appeared in \cite{GW}. In our notations, what they are considering is the generalized cohomology theory associated to (the loop space of) the spectrum $\mathbf{SVect}^{\times}$, whose homotopy groups are $\mathbb{Z}/2$ and $\mathds{k}^{\times}$. This generalized cohomology theory is sometimes called restricted supercohomology. We will make no use of this notion.} More precisely, as we have already recalled in the proof of Lemma~\ref{lem:kinvariantsZ2SVect} above, the homotopy groups of~$\mathbf{2SVect}^{\times}$ are given by $\mathbb{Z}/2$, $\mathbb{Z}/2$, $\mathds{k}^{\times}$, and supercohomology is shifted so that $SH^0(\mathrm{pt})=\mathds{k}^{\times}$. In particular, for any space $X$ and integer $n$, it follows from the Atiyah--Hirzebruch spectral sequence
\begin{gather*}
E_2^{i,j}=H^i\big(X;SH^j(\mathrm{pt})\big)\Rightarrow SH^{i+j}(\mathrm{B}G)
\end{gather*}
 that the group $SH^n(X)$ has a filtration with successive subquotients $E^{n,0}_{\infty}$, $E^{n-1,1}_{\infty}$, and $E^{n-2,2}_{\infty}$. But, $E^{n-2,2}_{\infty}$ is a subgroup of $E^{n-2,2}_{2}=H^{n-2}(X;\mathbb{Z}/2)$, $E^{n-1,1}_{\infty}$ is a subquotient of $E^{n-1,1}_{2}=H^{n-1}(X;\mathbb{Z}/2)$, and \smash{$E^{n,0}_{\infty}$} is a quotient of \smash{$E^{n,0}_{2}=H^n(X;\mathds{k}^{\times})$}, so that a class $\pi$ in $SH^n(X)$ may, by abusing notations, be written as a triple $(\alpha,\beta,\gamma)$ with $\alpha\in H^{n-2}(X;\mathbb{Z}/2)$, $\beta\in H^{n-1}(X;\mathbb{Z}/2)$, and $\gamma\in H^n(X;\mathds{k}^{\times})$.

Now, for the purposes of classifying fermionic strongly fusion 2-categories, we will also need to consider twisted supercohomology. More precisely, the spectrum $\mathbf{2SVect}^{\times}$ has a non-trivial space of automorphisms. We will only be interested in the subspace given by
\begin{gather*}
\mathscr{A}ut^{\rm tens}(\mathbf{2SVect}) \simeq \mathscr{A}ut^{\operatorname{br}}(\mathbf{SVect})\simeq \mathbb{Z}/2[1].
\end{gather*}
 Concretely, this corresponds to the symmetric monoidal natural autoequivalence $\varphi$ of the identity functor on $\mathbf{SVect}$ which takes the value $-1$ on purely odd vector spaces. Then, as is explained in \cite[Section~4.1]{NSS}, the action of the higher group $\mathbb{Z}/2[1]$ on the spectrum $\mathbf{2SVect}^{\times}$ is encoded by the fiber sequence of spaces
 \begin{equation}\begin{tikzcd}[sep=small]
\mathrm{B}^{n-2}\mathbf{2SVect}^{\times} \arrow[r] & {\mathrm{B}^{n-2}\mathbf{2SVect}^{\times}//(\mathbb{Z}/2[1])} \arrow[d] \\
&{\mathrm{B}\mathbb{Z}/2[1]=\mathbb{Z}/2[2]}\label{eq:canonicalbundle}
\end{tikzcd}\end{equation}
 for large enough $n$.

\begin{Definition}\label{def:twistedSH}
Let $X$ be a space equipped with an action $\varpi\colon X\rightarrow \mathbb{Z}/2[2]$, and write $P\rightarrow X$ for the bundle with fiber $\mathrm{B}^{n-2}\mathbf{2SVect}^{\times}$ obtained by pulling back (\ref{eq:canonicalbundle}) along $\varpi$. The $\varpi$-twisted $n$-th supercohomology of $X$ is the group $SH^{n+\varpi}(X):=\Gamma_X(P)$ of homotopy classes of sections of the bundle $P\rightarrow X$.
\end{Definition}

Just as we have explained above in the untwisted case, by abusing notations, we will also write classes in twisted supercohomology groups as triples.
We now identify the total space $\mathrm{B}^2\mathbf{2SVect}^{\times}\allowbreak//(\mathbb{Z}/ 2[1])$.

\begin{Lemma}\label{lem:supercohomolohyfibration}
The canonical fibration $\mathrm{B}^2\mathbf{2SVect}^{\times}//(\mathbb{Z}/2[1])\rightarrow \mathbb{Z}/2[2]$ is isomorphic to the map~$\mathrm{B}^2\mathscr{Z}(\mathbf{2SVect})^{\times}\!\!\rightarrow\! \mathbb{Z}/2[2]$ collapsing the connected component of the identity of $\mathscr{Z}(\mathbf{2SVect})$.
\end{Lemma}
\begin{proof}
As above, let $M$ denote an invertible object of $\mathscr{Z}(\mathbf{2SVect})$ that is not in the connected component of the identity. In order to identify the spaces $\mathrm{B}^2\mathbf{2SVect}^{\times}//(\mathbb{Z}/2[1])$ and $\mathrm{B}^2\mathscr{Z}(\mathbf{2SVect})^{\times}$, it is enough to show that the induced action of the invertible object $M$ in~$\mathscr{Z}(\mathbf{2SVect})$ on $\mathscr{Z}(\mathbf{2SVect})^0\simeq\mathbf{2SVect}$ is the canonical one.

The invertible object $M$ induces a braided monoidal 2-natural equivalence of $\mathbf{2SVect}$. More precisely, let us use $b$ to denote the braiding of $\mathscr{Z}(\mathbf{2SVect})$, which is an adjoint 2-natural equivalence equipped with a pseudo-inverse $b^{\bullet}$. Let us also write $u\colon I\simeq M\Box M$ for an adjoint equivalence witnessing that $M$ is invertible. Then, the 2-natural autoequivalence $t$ of the identity 2-functor on $\mathscr{Z}(\mathbf{2SVect})$ that is given on an object $X$ in $\mathscr{Z}(\mathbf{2SVect})$ by
\begin{gather*}
\smash{t_X\colon\ X\xrightarrow{X\Box u} X\Box M\Box M\xrightarrow{b_{X,M}\Box M} M\Box X\Box M\xrightarrow{M\Box b^{\bullet}_{X,M}} M\Box M\Box X\xrightarrow{u^{\bullet}\Box X}X},
\end{gather*}
 can be canonically upgraded to a braided monoidal 2-natural equivalence. In particular, we can restrict $t$ to a braided monoidal 2-natural equivalence of the identity 2-functor on $\mathbf{2SVect}$.

Now, as $\mathscr{A}ut^{\operatorname{br}}(\mathbf{2SVect})\simeq\mathscr{A}ut^{\rm tens}(\mathbf{2SVect})\simeq \mathbb{Z}/2[1]$, it is enough to check that this action is non-trivial. In order to see this, let $e$ denote the non-identity invertible 1-morphism in~$\Omega\mathbf{2SVect}=\mathbf{SVect}$. Then, it was explained in \cite[Section~3.1]{JFR} that the double braiding of the object $M$ and the 1-morphism $e$ is given by $b_{M,e}\cdot b_{e,M}=(-1) \operatorname{Id}_{M\Box e}$. This shows that~$\Omega t$ is the non-trivial braided monoidal autoequivalence of $\mathbf{SVect}$, and therefore concludes the proof.
\end{proof}

Combining the last lemma together with Theorem~\ref{thm:extensiontheory} yields the following result. We wish to point out that, relying on the as-of-yet incomplete theory of higher condensations \cite{GJF}, a~version of the classification of fermionic strongly fusion 2-categories has already been given in \cite[Section~V.D]{JF}. In the physics literature, an even earlier, albeit slightly incorrect, version appeared in \cite{LW}.

\begin{Theorem}\label{thm:classificationFSF2Cs}
Fermionic strongly fusion $2$-categories equipped with a faithful grading are classified by a finite group $G$ together with a class $\varpi$ in $H^2(\mathrm{B}G;\mathbb{Z}/2)$ and a class $\pi$ in $SH^{4+\varpi}(\mathrm{B}G)$.
\end{Theorem}
\begin{proof}
It follows from Theorem~\ref{thm:extensiontheory} that (faithfully) $G$-graded extensions of $\mathbf{2SVect}$ are classified by homotopy classes of maps
 \begin{gather*}
\mathrm{B}G\rightarrow \mathrm{B}\mathscr{B}{\rm r}\mathscr{P}{\rm ic}(\mathbf{2SVect})\simeq\mathrm{B}^2\mathscr{Z}(\mathbf{2SVect})^{\times}.
\end{gather*}
 The result then follows from the last lemma above. Namely, the class $\varpi$ is given by the composite~$\mathrm{B}G\rightarrow \mathrm{B}^2\mathscr{Z}(\mathbf{2SVect})^{\times}\rightarrow \mathrm{B}\mathbb{Z}/2[1]$, and endows $\mathrm{B}G$ with an action by $\mathbb{Z}/2[1]$. In addition, the data of a map $\mathrm{B}G\rightarrow \mathrm{B}^2\mathscr{Z}(\mathbf{2SVect})^{\times}$ lifting $\varpi$ is precisely that of a class in $SH^{4+\varpi}(\mathrm{B}G)$.
\end{proof}

\begin{Remark}
Without taking into account the faithful grading, fermionic strongly fusion 2-categories are classified by a finite group $G$ together with a class in $H^2(\mathrm{B}G;\mathbb{Z}/2)/\operatorname{Out}(G)$ and a class in $SH^{4+\varpi}(\mathrm{B}G)/\operatorname{Out}(G)$.
\end{Remark}

\begin{Remark}\label{rem:varpialphaunfolding}
As is clear from the proof, the finite group $G$ corresponds to the group of connected components of the fermionic strongly fusion 2-category $\mathfrak{C}$. We emphasize that the class~$\varpi$ in $H^2(\mathrm{B}G;\mathbb{Z}/2)$ is not the extension class determining the group of invertible objects $\operatorname{Inv}(\mathfrak{C})$ as a $\mathbb{Z}/2$ extension of $G$! This was first observed in the physics literature \cite{LW} and then given a~more mathematical treatment in \cite{JF2}, and transpires from Example~\ref{ex:FSF2CZ/2} below, but also from the result of Appendix~\ref{sec:computation}. Rather, the extension class is given by the bottom layer $\alpha$ of the class~${\pi=(\alpha,\beta,\gamma)}$ in $SH^{4+\varpi}(\mathrm{B}G)$. The class $\beta$ supplies the 1-morphisms witnessing associativity, and the class $\gamma$ gives the pentagonator. As for the class $\varpi$, it follows from the proof of Theorem~\ref{thm:classificationFSF2Cs} that it corresponds to the action of $G$ on $\mathfrak{C}^0=\mathbf{2SVect}$ by conjugation. Said differently, $\varpi$ encodes the data of the interchanger, or, equivalently, the 2-naturality of the associator. That these two pieces of coherence data are intimately related can be seen from~\cite[Proposition~4.2, equation~(A$\hat{a}$2)]{El1}. We expect that the action of $G$ on $\mathbf{2SVect}$ can be detected at the level of the Drinfeld center $\mathscr{Z}(\mathfrak{C})$ of $\mathfrak{C}$. More precisely, let us write \smash{$\big(\widetilde{G},z\big)$} for the central extension of~$G$ by~$\mathbb{Z}/2$ parameterised by $\varpi$, then let \smash{$\mathbf{Rep}\big(\widetilde{G},z\big)$} be the subcategory of super-representations of $\widetilde{G}$ on which $z$ acts as the parity automorphism. We expect that \smash{$\Omega\mathscr{Z}(\mathfrak{C})\simeq \mathbf{Rep}\big(\widetilde{G},z\big)$} as symmetric fusion 1-categories. We have checked this property for the fermionic strongly fusion 2-categories of Examples \ref{ex:actionandextension}, \ref{ex:actionnoextension} and \ref{ex:lin2group} below.
\end{Remark}

\begin{Remark}
Let us momentarily work over the field of real numbers $\mathbb{R}$. It is interesting to ask whether the classification of strongly fusion 2-categories remains valid at this level of generality. This turns out to be wrong, even in the bosonic case: The most general version of~\cite[Theorem~A]{JFY} does not hold over fields that are not algebraically closed. More precisely, it is not true that every fusion 2-category $\mathfrak{C}$ over $\mathbb{R}$ with $\Omega\mathfrak{C} = \mathbf{Vect}_{\mathbb{R}}$ is a group graded extension of~$\mathbf{2Vect}_{\mathbb{R}}$. The following counterexample was pointed out to us by Theo Johnson-Freyd. Let us consider the fusion 2-category $\mathbf{2Rep}_{\mathbb{R}}(\mathbb{Z}/3[2])$ of real 2-representations of the 2-group $\mathbb{Z}/3[2]$, that is, finite semisimple $\mathbb{R}$-linear 1-categories equipped with an action of $\mathbb{Z}/3[1]$. Over the complex numbers~$\mathbb{C}$, the notion of a 2-representation was first considered in \cite[Section 6]{El2} (see also \cite[Section 1.4.5]{DR} for a recent account). The underlying 2-category and fusion rules of the fusion 2-category $\mathbf{2Rep}_{\mathbb{R}}(\mathbb{Z}/3[2])$ are as depicted below.
 \begin{displaymath}
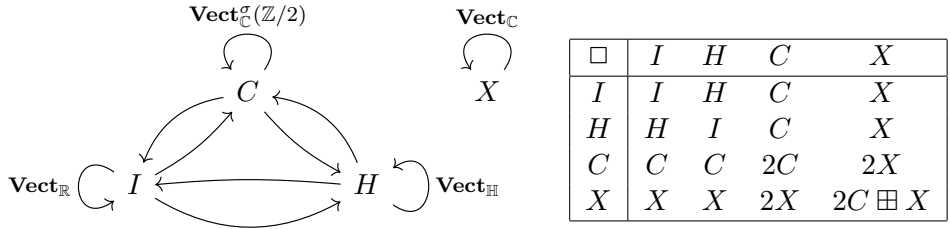
\begin{tabular}{c c}
$\begin{tikzcd}
& C \arrow[ld, bend right] \arrow[rd, bend right = 10] \arrow["\mathbf{Vect}^{\sigma}_{\mathbb{C}}(\mathbb{Z}/2)"', loop, distance=2em, in=125, out=55] & & X \arrow["\mathbf{Vect}_{\mathbb{C}}"', loop, distance=2em, in=125, out=55] \\
I \arrow[rr, bend right] \arrow[ru, bend right = 10] \arrow["\mathbf{Vect}_{\mathbb{R}}"', loop, distance=2em, in=215, out=145] & & H \arrow[lu, bend right] \arrow[ll, bend right = 5] \arrow["\mathbf{Vect}_{\mathbb{H}}"', loop, distance=2em, in=35, out=325] & \end{tikzcd}$&$\begin{tabular}{ |c| c c c c | }
 \hline
 $\Box$ & $I$ & $H$ & $C$ & $X$\\
 \hline
 $I$ & $I$ & $H$ & $C$ & $X$\\
 $H$ & $H$ & $I$ & $C$ & $X$\\
 $C$ & $C$ & $C$ & $2C$ & $2X$\\
 $X$ & $X$ & $X$ & $2X$ & $2C\boxplus X$\\
 \hline
\end{tabular}$
\end{tabular}\end{displaymath}
In particular, this shows that, unlike in the case of algebraically closed fields \cite[Example~1.4.22]{DR}, the fusion 2-categories $\mathbf{2Rep}_{\mathbb{R}}(\mathbb{Z}/3[2])$ and $\mathbf{2Vect}_{\mathbb{R}}(\mathbb{Z}/3)$ are not monoidally equivalent. Nevertheless, they become equivalent upon base extension to $\mathbb{C}$. This suggests that it is possible to classify real strongly fusion 2-categories by combining the known classification of strongly fusion 2-categories over an algebraically closed field of characteristic zero together with a 2-categorical version of the descent techniques of \cite{EG}.
\end{Remark}

\subsection{Examples}

We examine various special cases of the classification of fermionic strongly fusion 2-categories obtained above.

\begin{Example}\label{ex:topcocycleonly}
Let us take $G$ a finite group. Then, for any 4-cocycle $\gamma$ for $G$ with coefficients in $\mathds{k}^{\times}$, we can consider the fermionic strongly fusion 2-category $\mathbf{2SVect}\boxtimes \mathbf{2Vect}^{\gamma}(G)$. Their Drinfeld centers are completely understood thanks to \cite{KTZ} and \cite{JF2}, as taking Drinfeld centers commutes with 2-Deligne tensor products by \cite{D9}. Further, in the classification of Theorem~\ref{thm:classificationFSF2Cs} the corresponding data is $\varpi = \operatorname{triv}$, and the class $\pi$ in $SH^4(\mathrm{B}G)$ is given by $(\operatorname{triv}, \operatorname{triv}, \gamma)$. However, for a general group $G$, different 4-cocycles $\gamma$ may yield equivalent fusion 2-categories (see Example~\ref{ex:FSF2CZ2Z2} below). Finally, let us note that, when $G$ has odd order, these are all of the fermionic strongly fusion 2-categories. Namely, in this case, we have $H^2(\mathrm{B}G;\mathbb{Z}/2)\cong 0$ and~$SH^{4}(\mathrm{B}G)\cong H^4(\mathrm{B}G;\mathds{k}^{\times})$.
\end{Example}

\begin{Example}\label{ex:actionandextension}
Let \smash{$\big(\widetilde{G},z\big)$} be a finite super-group, that is a finite group $\widetilde{G}$ equipped with a~central element $z$ of order exactly 2. Let us write $\mathcal{B}$ for (any of) the braided Ising 1-categories. Such a braided fusion 1-category has a canonical $\mathbb{Z}/2$-grading $\mathcal{B}=\mathcal{B}_0\boxplus \mathcal{B}_1$ with $\mathcal{B}_0=\mathbf{SVect}$ and $\mathcal{B}_1=\mathbf{Vect}$. This allows us to consider $\mathcal{B}$ as a non-faithfully $\widetilde{G}$-crossed braided 1-category by setting $\mathcal{B}_e=\mathcal{B}_0$, $\mathcal{B}_z=\mathcal{B}_1$, $\mathcal{B}_g=0$ for any other $g\in\widetilde{G}$, and taking the trivial $\widetilde{G}$-action. Then, we can use \cite[Construction~2.1.23]{DR}, recalled above in Construction~\ref{con:gradedF2CfromCBF1C}, so as to obtain a monoidal 2-category $\widehat{\mathfrak{D}}$ whose group of objects is $\widetilde{G}$, and with $\operatorname{Hom}$-1-categories given by~\smash{$\operatorname{Hom}_{\widehat{\mathfrak{D}}}(g,h):=\mathcal{B}_{hg^{-1}}$}. We then get a fusion 2-category $\mathfrak{D}$ by taking the additive completion of~$\widehat{\mathfrak{D}}$. (In the specific case under consideration, the additive completion is the Cauchy completion.) Further, it is clear that $\operatorname{Inv}(\mathfrak{C})\cong \widetilde{G}$ and $\pi_0(\mathfrak{C})=\widetilde{G}/z$. The fermionic strongly fusion 2-category corresponding to $\widetilde{G}=\mathbb{Z}/4$ is depicted below in Figure~\ref{fig:exoticFSF2C}.

It is informative to understand how the fermionic strongly fusion 2-category $\mathfrak{D}$ constructed above fits into the classification of Theorem~\ref{thm:classificationFSF2Cs}. As was already pointed out, the corresponding group is $G:=\widetilde{G}/z$. Further, it follows from the construction of $\mathfrak{D}$ that the class~$\varpi$ in~$H^2(\mathrm{B}G;\mathbb{Z}/2)$ is the class corresponding to the central extension $\widetilde{G}$ of $G$. More precisely, one computes that $\Omega\mathscr{Z}(\mathfrak{D})\simeq\mathbf{Rep}\big(\widetilde{G},z\big)$. Finally, the class $\pi$ in $SH^{4+\varpi}(\mathrm{B}G)$ is given by $(\varpi, \operatorname{triv}, \operatorname{triv})$. In particular, the construction above does not depend on the choice of braided Ising 1-category~$\mathcal{B}$.
\end{Example}

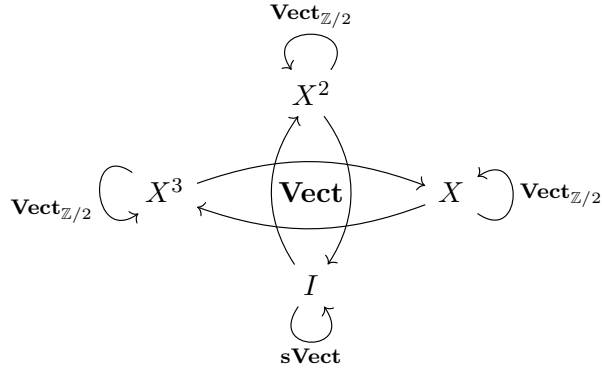
\begin{figure}[!hbt]\vspace{-5mm}
$$\begin{tikzcd}
& X^2 \arrow[dd, bend left=35] \arrow["\mathbf{Vect}_{\mathbb{Z}/2}"', loop, distance=2em, in=125, out=55] &\\
X^3 \arrow[rr, bend left = 20] \arrow["\mathbf{Vect}_{\mathbb{Z}/2}"', loop, distance=2em, in=215, out=145] & \mathbf{Vect}& X \arrow[ll, bend left = 20] \arrow["\mathbf{Vect}_{\mathbb{Z}/2}"', loop, distance=2em, in=35, out=325] \\
& I \arrow[uu, bend left=35] \arrow["\mathbf{sVect}"', loop, distance=2em, in=305, out=235] &
\end{tikzcd}$$
\caption{An exotic fermionic strongly fusion 2-category.}
\label{fig:exoticFSF2C}
\end{figure}

\begin{Example}\label{ex:actionnoextension}
Again, let us fix \smash{$\big(\widetilde{G},z\big)$} a finite super-group. Then, we may consider the symmetric fusion 1-category $\mathbf{Rep}\big(\widetilde{G},z\big)$ of finite dimensional super-representations of $G$ on which~$z$ acts by the parity automorphism. We write \smash{$\mathbf{2Rep}\big(\widetilde{G},z\big):=\mathbf{Mod}\big(\mathbf{Rep}\big(\widetilde{G},z\big)\big)$} for the corresponding (symmetric) fusion 2-category. Then, given that there is an essentially unique fiber functor \smash{$\mathbf{Rep}\big(\widetilde{G},z\big)\rightarrow\mathbf{SVect}$}, we can view $\mathbf{2SVect}$ as a finite semisimple module 2-category for~\smash{$\mathbf{2Rep}\big(\widetilde{G},z\big)$}. We write
$\mathfrak{C}:=\mathbf{\operatorname{End}}_{\mathbf{2Rep}(\widetilde{G},z)}(\mathbf{2SVect})
$
 for its Morita dual fusion 2-category. By~\cite[Lemma~3.2.1]{D9}, $\mathfrak{C}$ is a fermionic strongly fusion 2-category.

We claim that, in the formulation of Theorem~\ref{thm:classificationFSF2Cs} the data corresponding to the fusion 2-category $\mathfrak{C}$ is the finite group $G:=\widetilde{G}/z$, the class $\varpi$ in $H^2(\mathrm{B}G;\mathbb{Z}/2)$ is the class classifying the central extension $\widetilde{G}$ of $G$, and the class $\pi$ in $SH^{4+\varpi}(\mathrm{B}G)$ is the trivial one. In particular, we have~${\operatorname{Inv}(\mathfrak{C})\cong G\oplus\mathbb{Z}/2}$, so that $\mathfrak{C}$ is distinct from the fermionic strongly fusion 2-category~$\mathfrak{D}$ constructed in the previous example. In order to see this, observe that there is a~monoidal 2-functor~${\mathfrak{C}\rightarrow\mathbf{2SVect}}$. Namely, by construction, $\mathfrak{C}$ acts on the finite semisimple 2-category~$\mathbf{2SVect}$, i.e., there is a monoidal 2-functor $\mathbf{F}\colon\mathfrak{C}\rightarrow \mathbf{\operatorname{End}}(\mathbf{2SVect})\simeq\mathbf{Mod}(\mathcal{Z}(\mathbf{SVect}))$. But, this action commutes with the action of $\mathbf{2Rep}\big(\widetilde{G},z\big)$, so that the image of $\mathbf{F}$ must land in the monoidal sub-2-category $\mathbf{2SVect}$ of $\mathbf{Mod}(\mathcal{Z}(\mathbf{SVect}))$ as desired. This implies that the class~$\pi$ is trivial. On the other hand, it follows from \cite[Theorem~2.3.2]{D9} that $\Omega\mathscr{Z}(\mathfrak{C})\simeq \mathbf{Rep}\big(\widetilde{G},z\big)$, so that the class $\varpi$ is as claimed above.
\end{Example}

\begin{Example}\label{ex:lin2group}
Let $G$ be a finite group, $\varpi$ a 2-cocycle for $G$ with coefficients in $\mathbb{Z}/2$, and $\pi=(\alpha,\beta,\gamma)$ a class in $SH^{4+\varpi}(\mathrm{B}G)$. Provided that $\alpha=\operatorname{triv}$, then the corresponding fermionic strongly fusion 2-category is a fusion 2-category of twisted 2-group graded 2-vector spaces. More precisely, we can think of the 3-cocycle $\beta$ for $G$ with coefficients in $\mathbb{Z}/2$ as the Postnikov class of a finite 2-group $\mathcal{G}=\mathbb{Z}/2[1]\boldsymbol{\cdot}G[0]$. Given a 4-cocycle $\omega$ in $H^4(\mathrm{B}\mathcal{G},\mathds{k}^{\times})$, we can consider the fusion 2-category $\mathbf{2Vect}^{\omega}(\mathcal{G})$ of $\omega$-twisted $\mathcal{G}$-graded 2-vector spaces. This is explained in detail in \cite[Construction 2.1.16]{DR}. This is a fermionic strongly fusion 2-category exactly if the image of $\omega$ under the canonical map $H^4(\mathrm{B}\mathcal{G},\mathds{k}^{\times})\rightarrow H^4(\mathrm{B}^2\mathbb{Z}/2,\mathds{k}^{\times})$ is $(-1)^{\mathrm{Sq}^2}$. Namely, in this case, we have $\Omega\mathbf{2Vect}^{\omega}(\mathcal{G})\simeq\mathbf{SVect}$. Furthermore, every fermionic strongly fusion 2-category with~${\alpha=\operatorname{triv}}$ can be obtained via this construction. In particular, this includes all the fermionic strongly fusion 2-categories of Example~\ref{ex:topcocycleonly}.
\end{Example}

We now turn our attention to fermionic strongly fusion 2-categories whose group of connected components is a fixed 2-torsion group.

\begin{Example}\label{ex:FSF2CZ/2}
Let us examine the case $G=\mathbb{Z}/2^n$ with $n\geq 1$. Then, we have $H^2(\mathrm{B}\mathbb{Z}/2^n;\mathbb{Z}/2)\allowbreak\cong\mathbb{Z}/2$. On one hand, we have $SH^{4}(\mathrm{B}\mathbb{Z}/2^n)=0$ corresponding to the fermionic strongly fusion 2-category $\mathbf{2SVect}(\mathbb{Z}/2^n)$. Namely, we can consider the Atiyah--Hirzebruch spectral sequence
\begin{gather*}
E_2^{i,j}=H^i\big(\mathrm{B}\mathbb{Z}/2^n;SH^j(\mathrm{pt})\big)\Rightarrow SH^{i+j}(\mathrm{B}\mathbb{Z}/2^n)
\end{gather*}
 with corresponding $E_2$ page given by
\begin{gather*}\renewcommand{\arraystretch}{1.2}\begin{array}{c|ccccccc}
 j \\
 2 & \mathbb{Z}/2 & \mathbb{Z}/2 & \mathbb{Z}/2 & \mathbb{Z}/2 \\
 1 & \mathbb{Z}/2 & \mathbb{Z}/2 & \mathbb{Z}/2 & \mathbb{Z}/2 & \mathbb{Z}/2\\
 0 & \mathds{k}^{\times}& \mathbb{Z}/2^n & 0 & \mathbb{Z}/2^n & 0 & \mathbb{Z}/2^n \\
 \hline
 & 0 & 1 & 2 & 3 & 4 & 5 & \quad i.
\end{array}\end{gather*}

The $d_2$ differentials for this spectral sequence, for any finite group $G$ and with $\varpi=\operatorname{triv}$, are given by the $k$-invariants of $\mathbf{2SVect}^{\times}$, so that we have
\begin{equation}\label{eq:untwistedd2s}
d_2=\mathrm{Sq}^2\colon\ E^{i,2}_2\rightarrow E^{i+2,1}_2\qquad \textrm{and}\qquad d_2=(-1)^{\mathrm{Sq}^2}\colon\ E^{i,1}_2\rightarrow E^{i+2,0}_2,
\end{equation}
provided that $i\geq 1$, and where, as above, $t\mapsto (-1)^{t}$ is the homomorphism induced by $\mathbb{Z}/2\hookrightarrow \mathds{k}^{\times}$. In particular, the terms in degrees $(2,2)$ and $(3,1)$ are both killed so that $SH^{4}(\mathrm{B}\mathbb{Z}/2^n)=0$ as claimed.

On the other hand, using $\varpi$ to denote the non-trivial class in $H^2(\mathrm{B}\mathbb{Z}/2^{n},\mathbb{Z}/2)$, we claim that $SH^{4+\varpi}(\mathrm{B}\mathbb{Z}/2^{n})=\mathbb{Z}/2$. Namely, in that case, the $E_2$ page of the corresponding Atiyah--Hirzebruch spectral sequence is the same as above. However, the differentials \smash{$d_2\colon E^{i,j}_2\rightarrow E^{i-1,j+2}_2$} may be different. We do not know how to describe these differentials, and will therefore use other techniques in order to compute this group. We have to understand whether or not the groups in degrees $(2,2)$ and $(3,1)$ survive to the $E_{\infty}$ page. We assert that the group in degree $(2,2)$ does. Namely, this follows from the fact that we have exhibited two non-equivalent fermionic strongly fusion 2-categories $\mathfrak{C}$ as in Example~\ref{ex:actionnoextension} and $\mathfrak{D}$ as in Example~\ref{ex:actionandextension} with super-group $\bigl(\mathbb{Z}/2^{(n+1)},z\bigr)$ classified by $\varpi$ together with the classes $(\operatorname{triv},\operatorname{triv},\operatorname{triv})$ and $(\varpi,\operatorname{triv},\operatorname{triv})$. Now, let us assume that the group in degree $(3,1)$ survives. This would imply that there exists a class in $SH^{4+\varpi}(\mathrm{B}\mathbb{Z}/2^{n})$ of the form $(\operatorname{triv},\beta,\operatorname{triv})$ with non-trivial $\beta$. The corresponding fermionic strongly fusion 2-category then ought to be obtained via the construction of Example~\ref{ex:lin2group}. However, it follows by inspection that there are exactly two fermionic strongly fusion 2-category that can be obtained that way, which must then be $\mathbf{2SVect}(\mathbb{Z}/2^n)$ and $\mathfrak{C}$. We therefore find that the group in degree $(3,1)$ must be killed by the $d_2$ differential, thereby showing that $SH^{4+\varpi}(\mathrm{B}\mathbb{Z}/2^{n})=\mathbb{Z}/2$ as desired. We also wish to point out that it is expected that the two fusion 2-categories $\mathfrak{C}$ and $\mathfrak{D}$ have the same Drinfeld center. More specifically, if $\mathcal{B}$ is a~braided Ising fusion 1-category, then it is predicted that $\mathfrak{D}\boxtimes\mathbf{Mod}(\mathcal{B})$ is Morita equivalent to~$\mathfrak{C}$.
\end{Example}

\begin{Example}\label{ex:FSF2CZ2Z2}
We now take $G=\mathbb{Z}/2\oplus\mathbb{Z}/2$ with basis $a=(1,0)$ and $b=(0,1)$. In this case, we have
\[
H^2(\mathrm{B}G;\mathbb{Z}/2)\cong\mathbb{Z}/2\oplus\mathbb{Z}/2\oplus\mathbb{Z}/2.
\]
More precisely, if $c_1$, resp.\ $d_1$, denotes the elements of $H^1(\mathrm{B}G;\mathbb{Z}/2)$ that restricts non-trivially to $\langle a\rangle$, resp.\ $\langle b\rangle$, then $H^2(\mathrm{B}G;\mathbb{Z}/2)$ has a basis given by $c_1^2$, $d_1^2$, and $c_1d_1$. Let us consider the Atiyah--Hirzebruch spectral sequence
\begin{gather*}
E_2^{i,j}=H^i(\mathrm{B}G;SH^j(\mathrm{pt}))\Rightarrow SH^{i+j}(\mathrm{B}G)
\end{gather*}
corresponding to the trivial class $\operatorname{triv}$ in $H^2(\mathrm{B}G;\mathbb{Z}/2)$. Its $E_2$ page is given by
\begin{gather*}\renewcommand{\arraystretch}{1.2}\begin{array}{c|ccccccc}
 j \\
 2 & \mathbb{Z}/2 & \mathbb{Z}/2^{\oplus 2} & \mathbb{Z}/2^{\oplus 3} & \mathbb{Z}/2^{\oplus 4} \\
 1 & \mathbb{Z}/2 & \mathbb{Z}/2^{\oplus 2} & \mathbb{Z}/2^{\oplus 3} & \mathbb{Z}/2^{\oplus 4} & \mathbb{Z}/2^{\oplus 5}\\
 0 & \mathds{k}^{\times}& \mathbb{Z}/2^{\oplus 2} & \mathbb{Z}/2 & \mathbb{Z}/2^{\oplus 3} & \mathbb{Z}/2^{\oplus 2} & \mathbb{Z}/2^{\oplus 4} \\
 \hline
 & 0 & 1 & 2 & 3 & 4 & 5 & \quad i.
\end{array}\end{gather*}

The $d_2$ differentials are as described in the formulas in equation \eqref{eq:untwistedd2s}. Beyond the potential differential $d_3$, the other main difficulty lies in describing the differential $d_2=(-1)^{\mathrm{Sq}^2}$, and, more precisely, the image of $t\mapsto (-1)^{t}$.

In order to do so, consider the map of short exact sequences
 \begin{displaymath}\begin{tikzcd}[sep=small]
0 \arrow[r] & \mathbb{Z}/2 \arrow[r] \arrow[d, equal] & \mathbb{Z}/4 \arrow[r] \arrow[d] & \mathbb{Z}/2 \arrow[r] \arrow[d] & 0 \\
0 \arrow[r] & \mathbb{Z}/2 \arrow[r] & \mathds{k}^{\times} \arrow[r] & \mathds{k}^{\times} \arrow[r] & 0.
\end{tikzcd}\end{displaymath}
 For $i\geq 2$, this induces a commutative square
 \begin{displaymath}\begin{tikzcd}
H^{i-1}(\mathrm{B}G;\mathbb{Z}/2) \arrow[r, "\mathrm{Sq}^1"] \arrow[d] & H^i(\mathrm{B}G;\mathbb{Z}/2) \arrow[d, equal] \\
H^{i-1}(\mathrm{B}G;\mathds{k}^{\times}) \arrow[r, "\partial"'] & H^i(\mathrm{B}G;\mathbb{Z}/2).
\end{tikzcd}\end{displaymath}
 But, the bottom horizontal map is injective in the case $G=\mathbb{Z}/2\oplus\mathbb{Z}/2$ as every class in $H^i(\mathrm{B}G;\mathds{k}^{\times})$ is annihilated by the map on cohomology groups induced by $x\mapsto x^2$ on $\mathds{k}^{\times}$. This shows that the kernel of $t\mapsto (-1)^{t}$ consists exactly of those classes that are in the image of the Bockstein homomorphism $\mathrm{Sq}^1$.

In particular, we find that the group $SH^4(\mathrm{B}G)$ has order $16$. Namely, the $d_3$ differentials on the entries $(1,2)$ and $(2,2)$ vanish. For the latter, this is immediate because the entry in degree $(2,2)$ is completely killed by $d_2$. As for the former, this can be seen using the naturality of the Atiyah--Hirzebruch spectral sequence with respect to the various group homomorphisms~${\mathbb{Z}/2\hookrightarrow \mathbb{Z}/2\oplus \mathbb{Z}/2}$. A set of representatives for the classes in $SH^4(\mathrm{B}G)$ is therefore given by $(\operatorname{triv},\beta,\gamma)$ with $\gamma$ in $H^4(\mathrm{B}G;\mathds{k}^{\times})\cong \mathbb{Z}/2\oplus\mathbb{Z}/2$ arbitrary, and $\beta$ in the span of $c_1^2d_1$, $c_1d_1^2$ in~$H^3(\mathrm{B}G;\mathbb{Z}/2)$. The four fermionic strongly fusion 2-categories corresponding to $(\operatorname{triv},\operatorname{triv},\gamma)$ are the ones of Example~\ref{ex:topcocycleonly}. The others are different, but all arise via the construction discussed in Example~\ref{ex:lin2group}.

It would be interesting to compute the twisted supercohomology groups $SH^{4+\varpi}(\mathrm{B}G)$. Just as we have already seen in the preceding example, the main difficulty resides in describing the~$d_2$ differentials in the Atiyah--Hirzebruch spectral sequence in supercohomology.
\end{Example}

\appendix

\renewcommand{\thefootnote}{$*$}

\section[A computation in supercohomology. Appendix by Thibault Didier D\'ecoppet and Theo Johnson-Freyd]{A computation in supercohomology. Appendix by\\ Thibault Didier D\'ecoppet and Theo Johnson-Freyd\footnote{Perimeter Institute for Theoretical Physics, Waterloo, Ontario, Canada}
\renewcommand{\thefootnote}{}
\footnote{E-mail: \href{mailto:theojf@pitp.ca}{theojf@pitp.ca}}
\footnote{Department of Mathematics, Dalhousie University, Halifax, Nova Scotia, Canada}
\footnote{E-mail: \href{mailto:theojf@dal.ca}{theojf@dal.ca}}}\label{sec:computation}

Let $G = \mathbb{Z}/4\oplus \mathbb{Z}/4$. We claim that the canonical map $SH^4(\mathrm{B}G)\rightarrow H^2(\mathrm{B}G;\mathbb{Z}/2)$ is non-zero. Said differently, there exists a class $(\alpha,\beta,\gamma)$ in $SH^4(\mathrm{B}G)$ with $\alpha\neq \operatorname{triv}$. This proves that there exists a fermionic strongly fusion 2-category with trivial twist $\varpi$ whose group of invertible objects is a~non-trivial central extension of the group of connected components, thereby answering a~question of \cite{JFY}.

In order to prove the above claim, we will consider the Atiyah--Hirzebruch spectral sequence
\begin{gather*}
E^{i,j}_2=H^i(\mathrm{B}(\mathbb{Z}/4\oplus \mathbb{Z}/4);SH^j(\mathrm{pt}))\Rightarrow SH^{i+j}(\mathrm{B}(\mathbb{Z}/4\oplus \mathbb{Z}/4)),
\end{gather*}
 whose $E_2$ page is depicted below
\begin{equation*}\renewcommand{\arraystretch}{1.2}
 \begin{array}{c|ccccccc}
 j \\
 2 & \mathbb{Z}/2 & \mathbb{Z}/2^{\oplus 2} & \mathbb{Z}/2^{\oplus 3} & \mathbb{Z}/2^{\oplus 4} \\
 1 & \mathbb{Z}/2 & \mathbb{Z}/2^{\oplus 2} & \mathbb{Z}/2^{\oplus 3} & \mathbb{Z}/2^{\oplus 4} & \mathbb{Z}/2^{\oplus 5}\\
 0 & \mathds{k}^{\times}& \mathbb{Z}/4^{\oplus 2} & \mathbb{Z}/4 & \mathbb{Z}/4^{\oplus 3} & \mathbb{Z}/4^{\oplus 2} & \mathbb{Z}/4^{\oplus 4} \\
 \hline
 & 0 & 1 & 2 & 3 & 4 & 5 & \quad i.
\end{array}
\end{equation*}

The $d_2$ differentials are given by \eqref{eq:untwistedd2s}. It will be necessary to give names to various of the classes in the groups above. To this end, recall that there is an isomorphism of graded rings~$H^{\bullet}(\mathrm{B}\mathbb{Z}/4;\mathbb{Z}/2)\cong \mathbb{Z}/2[x_1,x_2]/x_1^2$, where $x_1$ has degree $1$ and $x_2$ has degree $2$. It then follows from the K\"unneth formula that
\begin{gather*}
H^{\bullet}(\mathrm{B}(\mathbb{Z}/4\oplus\mathbb{Z}/4);\mathbb{Z}/2)\cong \mathbb{Z}/2[x_1,y_1,x_2,y_2]/\big(x_1^2,y_1^2\big),
\end{gather*}
 where $x_1$, $y_1$ have degree $1$, $x_2$, $y_2$ have degree $2$, the classes $x_1$, $x_2$ are pulled back from $\mathbb{Z}/4\oplus 0$, and the classes $y_1$, $y_2$ are pulled back from $0\oplus \mathbb{Z}/4$. In particular, it follows that $E_3^{2,2}$, the kernel of the $d_2$ differential
 \begin{gather*}
 d_2=\mathrm{Sq}^2\colon\ E^{2,2}_2=H^2(\mathrm{B}(\mathbb{Z}/4\oplus\mathbb{Z}/4);\mathbb{Z}/2)\rightarrow E^{4,1}_2=H^4(\mathrm{B}(\mathbb{Z}/4\oplus\mathbb{Z}/4);\mathbb{Z}/2),
 \end{gather*}
 is spanned by the class $x_1y_1$. We also have to analyze $E_3^{5,0}$, the cokernel of the $d_2$ differential
 \begin{gather*}
 d_2=(-1)^{\mathrm{Sq}^2}\colon\ E^{3,1}_2=H^3(\mathrm{B}(\mathbb{Z}/4\oplus\mathbb{Z}/4);\mathbb{Z}/2)\rightarrow E^{5,0}_2=H^5(\mathrm{B}(\mathbb{Z}/4\oplus\mathbb{Z}/4);\mathds{k}^{\times}).
 \end{gather*}
 We assert that the classes
 \begin{equation}\label{eqn:2torsionclasses}
 (-1)^{x_1x_2^2},\qquad (-1)^{y_1x_2^2},\qquad (-1)^{x_1y_2^2},\qquad (-1)^{y_1y_2^2}
 \end{equation}
 are linearly independent in $H^5(\mathrm{B}(\mathbb{Z}/4\oplus\mathbb{Z}/4);\mathds{k}^{\times})$ and span the image of $d_2$. Namely, it follows from the short exact sequence of coefficients $0\rightarrow\mathbb{Z}/2\rightarrow\mathds{k}^{\times}\rightarrow\mathds{k}^{\times}\rightarrow 0$, that the image of the map $H^5(\mathrm{B}(\mathbb{Z}/4\oplus\mathbb{Z}/4);\mathbb{Z}/2)\rightarrow H^5(\mathrm{B}(\mathbb{Z}/4\oplus\mathbb{Z}/4);\mathds{k}^{\times})$ is precisely the subgroup~${\mathbb{Z}/2^{\oplus 4}\subseteq \mathbb{Z}/4^{\oplus 4}}$, so that $E_3^{5,0}\cong \mathbb{Z}/2^{\oplus 4}$. Moreover, as $(-1)^{\mathrm{Sq}^1}$ is the trivial map, this image is indeed generated by the classes in \eqref{eqn:2torsionclasses}: We can use this to give names to generators of the group ${H^5(\mathrm{B}(\mathbb{Z}/4\oplus\mathbb{Z}/4);\mathds{k}^{\times})}$. We use $\sqrt{x}$ with $x$ one of the classes in \eqref{eqn:2torsionclasses} to denote a class in $H^5(\mathrm{B}(\mathbb{Z}/4\oplus\mathbb{Z}/4);\mathds{k}^{\times})$ such that $\sqrt{x}\cdot \sqrt{x} = x$. By definition, the classes~$\sqrt{x}$ for~$x$ in \eqref{eqn:2torsionclasses} generate $H^5(\mathrm{B}(\mathbb{Z}/4\oplus\mathbb{Z}/4);\mathds{k}^{\times})$. Moreover, their images under the quotient map~${H^5(\mathrm{B}(\mathbb{Z}/4\oplus\mathbb{Z}/4);\mathds{k}^{\times})\twoheadrightarrow E_3^{5,0}\cong\mathbb{Z}/2^{\oplus 2}}$ are independent of any choices and generate the target. This means that the image of the class $x_1y_1$ under the $d_3$ differential \smash{$d_3\colon E^{2,2}_3\rightarrow E^{5,0}_3$} can be expressed as a linear combination of the classes $\sqrt{x}$ with $x$ as in \eqref{eqn:2torsionclasses}.

It remains to argue that the class $x_1y_1$ in $E_3^{2,2}\subseteq H^2(\mathrm{B}(\mathbb{Z}/4\oplus\mathbb{Z}/4);\mathbb{Z}/2)$ survives the $d_3$ differential \smash{$d_3\colon E_3^{2,2}\rightarrow E_3^{5,0}$}. In order to do so, we will use the naturality of the Atiyah--Hirzebruch spectral sequence. More precisely, we consider the epimorphism $f\colon \mathbb{Z}/4\oplus\mathbb{Z}/4\twoheadrightarrow \mathbb{Z}/4\oplus\mathbb{Z}/2$, the Atiyah--Hirzebruch spectral sequence
 \begin{gather*}
\widetilde{E}^{i,j}_2=H^i(\mathrm{B}(\mathbb{Z}/4\oplus \mathbb{Z}/2);SH^j(\mathrm{pt}))\Rightarrow SH^{i+j}(\mathrm{B}(\mathbb{Z}/4\oplus \mathbb{Z}/2)),
\end{gather*}
 together with the induced map of spectral sequences \smash{$f^*\colon\widetilde{E}^{i,j}_2\rightarrow E^{i,j}_2$}. It follows from the K\"unneth formula that
 \begin{gather*}
 H^{\bullet}(\mathrm{B}(\mathbb{Z}/4\oplus\mathbb{Z}/2);\mathbb{Z}/2)\cong \mathbb{Z}/2[x_1,z_1,x_2]/x_1^2,
 \end{gather*}
 where $x_1$, $z_1$ have degree $1$, $x_2$ has degree $2$, the classes $x_1$, $x_2$ are pulled back from $\mathbb{Z}/4\oplus 0$, and the class $z_1$ is pulled back from $0\oplus \mathbb{Z}/2$. Under the pullback map $f^*$, we have $x_1\mapsto x_1$, $x_2\mapsto x_2$, and $z_1\mapsto y_1$. In particular, observe that the class $x_1y_1$ is in the image of the pullback $f^*$. Further, we have that $\widetilde{E}_3^{2,2}$, the kernel of $d_2=\mathrm{Sq}^2\colon H^2(\mathrm{B}(\mathbb{Z}/4\oplus\mathbb{Z}/2);\mathbb{Z}/2)\rightarrow H^4(\mathrm{B}(\mathbb{Z}/4\oplus\mathbb{Z}/2);\mathbb{Z}/2),$ is spanned by the class $x_1y_1$. Moreover, using an argument similar to the one given above, we find that the image of $d_2=(-1)^{\mathrm{Sq}^2}\colon H^3(\mathrm{B}(\mathbb{Z}/4\oplus\mathbb{Z}/2);\mathbb{Z}/2)\rightarrow H^5(\mathrm{B}(\mathbb{Z}/4\oplus\mathbb{Z}/2);\mathds{k}^{\times})$ is~$\mathbb{Z}/2^{\oplus 4}\subseteq \mathbb{Z}/4\oplus\mathbb{Z}/2^{\oplus 3}\cong H^5(\mathrm{B}(\mathbb{Z}/4\oplus\mathbb{Z}/2);\mathds{k}^{\times})$ and is generated by the classes \begin{gather*}
 (-1)^{x_1x_2^2},\qquad (-1)^{z_1x_2^2},\qquad (-1)^{z_1^3x_2},\qquad (-1)^{z_1^5}.
 \end{gather*}
 It follows from the naturality of the K\"unneth formula that there exists a class \smash{$\sqrt{(-1)^{x_1x_2^2}}$} in $H^5(\mathrm{B}(\mathbb{Z}/4\oplus\mathbb{Z}/2);\mathds{k}^{\times})$ such that
 \[
 \sqrt{(-1)^{x_1x_2^2}}\cdot \sqrt{(-1)^{x_1x_2^2}} = (-1)^{x_1x_2^2}.
 \]
 In particular, the image of \smash{$\sqrt{(-1)^{x_1x_2^2}}$} under the quotient map $H^5(\mathrm{B}(\mathbb{Z}/4\oplus\mathbb{Z}/2);\mathds{k}^{\times})\twoheadrightarrow \widetilde{E}^{5,0}_3\cong \mathbb{Z}/2$ is a generator. Moreover, as \smash{$f^*\big((-1)^{x_1x_2^2}\big) = (-1)^{x_1x_2^2}$} by naturality of $t\mapsto (-1)^t$, we must have
 \[
 f^*(\sqrt{(-1)^{x_1x_2^2}}) = \sqrt{(-1)^{x_1x_2^2}}
 \]
 in $E^{5,0}_3$. But, if the $d_3$ differential $d_3\colon\widetilde{E}^{2,2}_3\rightarrow \widetilde{E}^{5,0}_3$ were non-zero, we would have
 \[d_3(x_1z_1) = \sqrt{(-1)^{x_1x_2^2}},
 \]
 and, by naturality of the Atiyah--Hirzebruch spectral sequence, we would therefore have that
 \[
 d_3(x_1y_1) = \sqrt{(-1)^{x_1x_2^2}}
 \]
in $E^{5,0}_3$. On the other hand, we can also consider the epimorphism $g\colon\mathbb{Z}/4\oplus\mathbb{Z}/4\twoheadrightarrow \mathbb{Z}/2\oplus\mathbb{Z}/4$, and, by an analogous argument, conclude that, provided that $d_3(x_1y_1)$ is non-zero, we would have \[d_3(x_1y_1)=\sqrt{(-1)^{y_1y_2^2}}.\] This shows that we must have $d_3(x_1y_1) = 0$, which concludes the proof of the claim.

\subsection*{Acknowledgements}

I am particularly indebted towards Theo Johnson-Freyd and David Reutter for sharing some of the ideas of their proof of the completely general version of group graded extension theory, which have inspired our proof of Theorem \ref{thm:extensiontheory}, and towards Matthew Yu for help with the cohomology computations of Section~\ref{sec:FSF2Cs}. I would also like to thank the referees for suggesting many invaluable improvements and clarifications. This work was supported in part by the Simons Collaboration on Global Categorical Symmetries.

\pdfbookmark[1]{References}{ref}
\LastPageEnding


\begin{thebibliography}{99}
\footnotesize\itemsep=0pt

\bibitem{Cr}
Crans S.E., Generalized centers of braided and sylleptic monoidal
 {$2$}-categories, \href{https://doi.org/10.1006/aima.1998.1720}{\textit{Adv. Math.}} \textbf{136} (1998), 183--223.

\bibitem{Cui}
Cui S.X., Four dimensional topological quantum field theories from
 {$G$}-crossed braided categories, \href{https://doi.org/10.4171/qt/128}{\textit{Quantum Topol.}} \textbf{10} (2019),
 593--676, \href{https://arxiv.org/abs/1610.07628}{arXiv:1610.07628}.

\bibitem{DN}
Davydov A., Nikshych D., Braided {P}icard groups and graded extensions of
 braided tensor categories, \href{https://doi.org/10.1007/s00029-021-00670-1}{\textit{Selecta Math.~(N.S.)}} \textbf{27} (2021),
 65, 87~pages, \href{https://arxiv.org/abs/2006.08022}{arXiv:2006.08022}.

\bibitem{D1}
D\'ecoppet T.D., Multifusion categories and finite semisimple 2-categories,
 \href{https://doi.org/10.1016/j.jpaa.2022.107029}{\textit{J.~Pure Appl. Algebra}} \textbf{226} (2022), 107029, 16~pages,
 \href{https://arxiv.org/abs/2012.15774}{arXiv:2012.15774}.

\bibitem{D2}
D\'ecoppet T.D., Weak fusion 2-categories, \textit{Cah. Topol. G\'eom.
 Diff\'er. Cat\'eg.} \textbf{63} (2022), 3--24, \href{https://arxiv.org/abs/2103.15150}{arXiv:2103.15150}.

\bibitem{D8}
D\'ecoppet T.D., The {M}orita theory of fusion 2-categories, \textit{High.
 Struct.} \textbf{7} (2023), 234--292, \href{https://arxiv.org/abs/2208.08722}{arXiv:2208.08722}.

\bibitem{D9}
D\'ecoppet T.D., Drinfeld centers and {M}orita equivalence classes of fusion
 2-categories, \textit{Compos. Math.}, {t}o appear, \href{https://arxiv.org/abs/2211.04917}{arXiv:2211.04917}.

\bibitem{D4}
D\'ecoppet T.D., Finite semisimple module 2-categories, \textit{Selecta Math.~(N.S.)}, {t}o appear, \href{https://arxiv.org/abs/2107.11037}{arXiv:2107.11037}.

\bibitem{D10}
D\'ecoppet T.D., On the dualizability of fusion 2-categoriess, \textit{Quantum Topol.}, {t}o appear, \href{https://arxiv.org/abs/2311.16827}{arXiv:2311.16827}.

\bibitem{DY23}
D\'ecoppet T.D., Yu M., Fiber 2-functors and {T}ambara--{Y}amagami fusion
 2-categories, \href{https://arxiv.org/abs/2306.08117}{arXiv:2306.08117}.

\bibitem{DR}
Douglas C.L., Reutter D.J., Fusion 2-categories and a state-sum invariant for
 4-manifolds, \href{https://arxiv.org/abs/1812.11933}{arXiv:1812.11933}.

\bibitem{El1}
Elgueta J., Cohomology and deformation theory of monoidal 2-categories.~{I},
 \href{https://doi.org/10.1016/S0001-8708(03)00078-1}{\textit{Adv. Math.}} \textbf{182} (2004), 204--277, \href{https://arxiv.org/abs/math.QA/0204099}{arXiv:math.QA/0204099}.

\bibitem{El2}
Elgueta J., Representation theory of 2-groups on {K}apranov and {V}oevodsky's
 2-vector spaces, \href{https://doi.org/10.1016/j.aim.2006.11.010}{\textit{Adv. Math.}} \textbf{213} (2007), 53--92,
 \href{https://arxiv.org/abs/math.CT/0408120}{arXiv:math.CT/0408120}.

\bibitem{EG}
Etingof P., Gelaki S., Descent and forms of tensor categories, \href{https://doi.org/10.1093/imrn/rnr119}{\textit{Int.
 Math. Res. Not.}} \textbf{2012} (2012), 3040--3063, \href{https://arxiv.org/abs/1102.0657}{arXiv:1102.0657}.

\bibitem{EGNO}
Etingof P., Gelaki S., Nikshych D., Ostrik V., Tensor categories, \textit{Math.
 Surveys Monogr}, Vol.~205, \href{https://doi.org/10.1090/surv/205}{American Mathematical Society}, Providence, RI,
 2015.

\bibitem{ENO2}
Etingof P., Nikshych D., Ostrik V., Fusion categories and homotopy theory,
 \href{https://doi.org/10.4171/QT/6}{\textit{Quantum Topol.}} \textbf{1} (2010), 209--273, \href{https://arxiv.org/abs/0909.3140}{arXiv:0909.3140}.

\bibitem{GJF0}
Gaiotto D., Johnson-Freyd T., Symmetry protected topological phases and
 generalized cohomology, \href{https://doi.org/10.1007/jhep05(2019)007}{\textit{J.~High Energy Phys.}} \textbf{2019} (2019),
 007, 34~pages, \href{https://arxiv.org/abs/1712.07950}{arXiv:1712.07950}.

\bibitem{GJF}
Gaiotto D., Johnson-Freyd T., Condensations in higher categories,
 \href{https://arxiv.org/abs/1905.09566}{arXiv:1905.09566}.

\bibitem{GW}
Gu Z.-C., Wen X.-G., Symmetry-protected topological orders for interacting
 fermions: Fermionic topological nonlinear {$\sigma$} models and a special
 group supercohomology theory, \href{https://doi.org/10.1103/PhysRevB.90.115141}{\textit{Phys. Rev.~B}} \textbf{90} (2014),
 115141, 59~pages, \href{https://arxiv.org/abs/1201.2648}{arXiv:1201.2648}.

\bibitem{JF}
Johnson-Freyd T., On the classification of topological orders, \href{https://doi.org/10.1007/s00220-022-04380-3}{\textit{Comm.
 Math. Phys.}} \textbf{393} (2022), 989--1033, \href{https://arxiv.org/abs/2003.06663}{arXiv:2003.06663}.

\bibitem{JF2}
Johnson-Freyd T., (3+1){D} topological orders with only a
 $\mathbb{Z}/2$-charged particle, \textit{Comm. Math. Phys.}, {t}o appear,
 \href{https://arxiv.org/abs/2011.11165}{arXiv:2011.11165}.

\bibitem{JFR3}
Johnson-Freyd T., Reutter D.J., Super-duper vector spaces~{I}, 2023, available
 at
 \url{https://homepages.uni-regensburg.de/~lum63364/ConferenceFFT/Reutter.pdf}.

\bibitem{JFR}
Johnson-Freyd T., Reutter D., Minimal nondegenerate extensions,
 \href{https://doi.org/10.1090/jams/1023}{\textit{J.~Amer. Math. Soc.}} \textbf{37} (2024), 81--150,
 \href{https://arxiv.org/abs/2105.15167}{arXiv:2105.15167}.

\bibitem{JFY}
Johnson-Freyd T., Yu M., Fusion 2-categories with no line operators are
 grouplike, \href{https://doi.org/10.1017/S0004972721000095}{\textit{Bull. Aust. Math. Soc.}} \textbf{104} (2021), 434--442,
 \href{https://arxiv.org/abs/2010.07950}{arXiv:2010.07950}.

\bibitem{JPR}
Jones C., Penneys D., Reutter D., A 3-categorical perspective on {$G$}-crossed
 braided categories, \href{https://doi.org/10.1112/jlms.12687}{\textit{J.~Lond. Math. Soc.}} \textbf{107} (2023),
 333--406, \href{https://arxiv.org/abs/2009.00405}{arXiv:2009.00405}.

\bibitem{KT}
Kapustin A., Thorngren R., Fermionic {SPT} phases in higher dimensions and
 bosonization, \href{https://doi.org/10.1007/jhep10(2017)080}{\textit{J.~High Energy Phys.}} \textbf{2017} (2017), no.~10,
 080, 48~pages, \href{https://arxiv.org/abs/1701.08264}{arXiv:1701.08264}.

\bibitem{KTZ}
Kong L., Tian Y., Zhou S., The center of monoidal 2-categories in {(3+1)D}
 {D}ijkgraaf--{W}itten theory, \href{https://doi.org/10.1016/j.aim.2019.106928}{\textit{Adv. Math.}} \textbf{360} (2020),
 106928, 25~pages, \href{https://arxiv.org/abs/1905.04644}{arXiv:1905.04644}.

\bibitem{LKW}
Lan T., Kong L., Wen X.-G., Classification of {(3+1)D} bosonic topological
 orders~{(I)}: The case when pointlike excitations are all bosons,
 \href{https://doi.org/10.1103/PhysRevX.8.021074}{\textit{Phys. Rev.~X}} \textbf{8} (2018), 021074, 24~pages,
 \href{https://arxiv.org/abs/1704.04221}{arXiv:1704.04221}.

\bibitem{LW}
Lan T., Wen X.-G., Classification of {(3+1)D} bosonic topological orders~{(II)}:
 The case when some pointlike excitations are fermions, \href{https://doi.org/10.1103/PhysRevX.9.021005}{\textit{Phys. Rev.~X}}
 \textbf{9} (2019), 021005, 37~pages, \href{https://arxiv.org/abs/1801.08530}{arXiv:1801.08530}.

\bibitem{L2}
Lurie J., Higher algebra, 2017, available at
 \url{https://www.math.ias.edu/~lurie/papers/HA.pdf}.

\bibitem{NSS}
Nikolaus T., Schreiber U., Stevenson D., Principal {$\infty$}-bundles: general
 theory, \href{https://doi.org/10.1007/s40062-014-0083-6}{\textit{J.~Homotopy Relat. Struct.}} \textbf{10} (2015), 749--801,
 \href{https://arxiv.org/abs/1207.0248}{arXiv:1207.0248}.

\bibitem{Pstr}
Pstr\c{a}gowski P., On dualizable objects in monoidal bicategories,
 \textit{Theory Appl. Categ.} \textbf{38} (2022), 257--310,
 \href{https://arxiv.org/abs/1411.6691}{arXiv:1411.6691}.

\bibitem{San}
Sanford S., Fusion categories over non-algebraically closed fields, Ph.D.~Thesis, {I}ndiana {U}niversity, 2022, available at
 \url{https://www.proquest.com/docview/2715420938/ACDC80FE6AEA477EPQ/1}.

\bibitem{SP}
Schommer-Pries C.J., The classification of two-dimensional extended topological
 field theories, Ph.D.~Thesis, {U}niversity of California, Berkeley, 2009,
 \href{https://arxiv.org/abs/1112.1000}{arXiv:1112.1000}.

\bibitem{WG}
Wang Q.-R., Gu Z.-C., Towards a complete classification of symmetry-protected
 topological phases for interacting fermions in three dimensions and a general
 group supercohomology theory, \href{https://doi.org/10.1103/PhysRevX.8.011055}{\textit{Phys. Rev.~X}} \textbf{8} (2018),
 011055, 29~pages, \href{https://arxiv.org/abs/1703.10937}{arXiv:1703.10937}.

\end{thebibliography}
\end{document}